\documentclass[11pt]{amsart}

\usepackage[margin=1.15in]{geometry}
\usepackage{amsmath,amssymb,amsfonts,amsthm}
\DeclareMathOperator{\pd}{pd}
\DeclareMathOperator{\depth}{depth}
\usepackage{mathtools}
\usepackage{mathrsfs}
\usepackage{enumitem}
\usepackage{tikz-cd}
\usepackage[colorlinks=true,citecolor=blue,linkcolor=blue,urlcolor=blue]{hyperref}
\usepackage{doi}

\newcommand{\cP}{\mathcal{P}}
\DeclareMathOperator{\Tor}{Tor}
\DeclareMathOperator{\Unip}{Unip}

\newcommand{\cU}{\mathcal{U}}
\newcommand{\FF}{\mathcal{F}}

\newcommand{\cO}{\mathcal{O}}

\newcommand{\OA}{\mathcal{O}_A}
\newcommand{\wA}{\widehat{A}}

\newtheorem{theorem}{Theorem}[section]

\newtheorem{proposition}[theorem]{Proposition}
\newtheorem{lemma}[theorem]{Lemma}
\newtheorem{corollary}[theorem]{Corollary}

\newtheorem{claim}[theorem]{Claim}

\theoremstyle{definition}
\newtheorem{definition}[theorem]{Definition}

\theoremstyle{remark}
\newtheorem{remark}[theorem]{Remark}

\newcommand{\OO}{\mathcal O}
\newcommand{\Pic}{\operatorname{Pic}}

\newcommand{\Ext}{\operatorname{Ext}}

\newcommand{\rk}{\operatorname{rk}}

\title[Existence of ACM Bundles on Abelian Varieties]
{Existence of ACM Bundles on Polarized Abelian Varieties}

\author{Soham Mondal}
\address{Department of Mathematics, IISER Berhampur, Ganjam, Odisha-760003, India}
\email{getsoham1@gmail.com}

\author{Pabitra Barik}
\address{Department of Mathematics, IISER Berhampur, Ganjam, Odisha-760003, India}
\email{pabitra@iiserbpr.ac.in}

\subjclass[2020]{14J60, 14F08, 16G60, 14K05}

\keywords{Abelian varieties, ACM sheaves, Ulrich bundles, Fourier--Mukai transform, Picard variety, Cohomological vanishing}

\begin{document}

\begin{abstract}
We study arithmetically Cohen--Macaulay bundles on polarized
abelian varieties of dimension at least two. We construct
explicit indecomposable families in every rank, compute their
cohomology, and establish wild representation type via the
Fourier--Mukai correspondence. We also obtain slope-stable
ACM bundles from simple semihomogeneous bundles, with ranks
determined by the polarization type and rational slope.
\end{abstract}

\maketitle

\section{Introduction}

An important theme in algebraic geometry is to understand a projective variety through the vector bundles it carries. Among these, two closely related classes have received a great deal of attention: arithmetically Cohen--Macaulay (ACM) bundles and Ulrich bundles. If $(A,L)$ is a polarized projective variety, a vector bundle $E$ is ACM with respect to $L$ if
\[
H^i(A,E\otimes L^{\otimes m})=0
\quad\text{for all } m\in \mathbb Z \text{ and } 0<i<\dim A.
\]
Ulrich bundles are the ACM bundles with the strongest possible cohomological behavior, and they often serve as the most rigid and useful objects in this circle of ideas.

These bundles matter for several reasons. In commutative algebra, ACM bundles correspond to maximal Cohen--Macaulay modules over homogeneous coordinate rings, so they reflect subtle properties of singularities, resolutions, and syzygies; see \cite{EisenbudSyzygies}. In algebraic geometry, Ulrich bundles became especially prominent through the work of Eisenbud and Schreyer (with an appendix by Weyman), who showed that they give determinantal and Pfaffian expressions for Chow forms; see  \cite{EisenbudSchreyer}. Beauville's papers and survey give a clear picture of the importance of Ulrich bundles and of the main existence problems around them; see \cite{BeauvilleUlrich,BeauvilleSurvey}.

Two broad questions organize much of the subject. The first is the existence problem: given a polarized projective variety, does it carry an ACM bundle or an Ulrich bundle? The second is the classification problem: if such bundles exist, how many indecomposable ones are there, and how complicated are their families? A useful language for the second question comes from representation theory. Following the now standard viewpoint, one studies the ACM representation type of a polarized variety, which may be finite, tame, or wild depending on the complexity of its indecomposable ACM bundles; see, for example, \cite{FaenziPonsLlopis}. For non-ACM projective schemes of dimension at least two, Pons-Llopis proved, under hypotheses on $\mathcal O_X(1)$, the existence of arbitrarily large-dimensional families of pairwise non-isomorphic
indecomposable ACM vector bundles with respect to $\mathcal O_X(\ell)$ for every $\ell\geq 3$ \cite{PL19}. This perspective has proved effective in measuring how complicated the geometry of ACM bundles really is.

Much of the existing literature concerns varieties whose geometry already favors cohomological vanishing, such as projective spaces, quadrics, Grassmannians, and several classes of Fano-type varieties. Abelian varieties lie in a very different direction. If $A$ is an abelian variety of dimension $g\ge 2$, then
\[
H^i(A,\mathcal O_A)\neq 0
\quad\text{for every }0\le i\le g,
\]
so the structure sheaf itself is not ACM. This makes abelian varieties a delicate testing ground for the theory. In dimension two, Beauville constructed rank-two Ulrich bundles on abelian surfaces; see \cite{BeauvilleUlrich}. Beyond such cases, however, the existence and structure of ACM bundles on polarized abelian varieties have remained much less transparent.

The starting point of this paper is an acyclicity statement with strong
consequences. Let $(A,L)$ be a polarized abelian variety over an
algebraically closed field of characteristic zero, and let
\[
P \in \operatorname{Pic}^{0}(A)\setminus\{\mathcal{O}_{A}\}.
\]
Using the Fourier--Mukai transform, we prove that $P$ is acyclic:
\[
H^{i}(A,P)=0
\qquad
\text{for all } 0\leq i\leq \dim A.
\]
Notice that $P$ is a degenerate line bundle, so this vanishing is not an application of the index theorem for nondegenerate line bundles. The Fourier--Mukai calculation treats the untwisted case. For $m>0$, the line bundle $P\otimes L^{\otimes m}$ is ample and therefore has no higher cohomology. For $m<0$, the required intermediate-cohomology vanishing follows from
Serre duality and the triviality of the canonical bundle of $A$. Combining these three cases, we prove that every nontrivial element of $\operatorname{Pic}^{0}(A)$ is ACM with respect to every ample polarization.

In fact, we determine the full cohomology table of $P\otimes L^{\otimes m}$ for all integers $m$. As consequences, we show that $\mathcal O_A$ is the unique non-ACM element of $\operatorname{Pic}^{0}(A)$ when $g\ge 2$, that $\operatorname{Pic}^{0}(A)\setminus\{\mathcal O_A\}$ gives a natural $g$-dimensional family of rank-one ACM line bundles, and that none of these line bundles is Ulrich, except in dimension one after one positive twist. We also show that the ACM complexity of $(A,L)$ is equal to one.

The rank-one case is only the beginning. Fix a nontrivial line bundle $P\in \operatorname{Pic}^{0}(A)$. For every $g\ge 2$ and every rank $r\ge 1$, we construct an indecomposable ACM vector bundle $E_r$ of rank $r$ by iterated non-split extensions starting from $P$. The main difficulty is to prove that suitable nontrivial extension classes exist at every step and that the resulting bundles remain indecomposable. To do this, we use Mukai's Fourier--Mukai transform \cite{MukaiDuality,MukaiSymmetry} and relate the extension problem to finite-length modules over a regular local ring.

The Fourier--Mukai transform identifies unipotent bundles on \(A\) with finite-length modules over the completed local ring of the dual abelian variety at its origin. For \(g\geq 2\), we use this equivalence to construct an exact representation embedding from finite-dimensional
\(k\langle u,v\rangle\)-modules into \(\operatorname{ACM}(A,L)\) for every ample \(L\). The embedding preserves indecomposability and reflects isomorphism classes. In Corollary~\ref{cor:large_ACM_families} we use its construction to obtain algebraic families of indecomposable ACM bundles over parameter spaces
of every positive dimension. Compare Pons-Llopis~\cite{PL19}, where wildness of non-ACM polarized schemes is obtained with respect to $\mathcal O_X(l)$, $l\ge3$. Thus, polarized abelian varieties support not only explicit ACM bundles in every rank, but also a maximally complicated ACM theory from the point of view of classification. This places abelian varieties of dimension two or greater firmly on the wild side of the representation spectrum, contrasting sharply with the tame behavior observed in certain varieties of minimal degree, such as quartic scrolls \cite{faenzi2017surfaces}.

The bundles obtained by successive extensions of \(P\) are strictly
slope-semistable in rank \(r\geq 2\). For each such rank, the explicit
\v{C}ech construction of Theorem~\ref{thm:moduli-explicit} gives a \((2g-1)\)-dimensional
algebraic family with pairwise non-isomorphic fibres, parametrized by
\[
    \bigl(\operatorname{Pic}^0(A)\setminus\{\mathcal O_A\}\bigr)
    \times \mathbb P\bigl(H^1(A,\mathcal O_A)\bigr).
\]
Its members \(E_r\) have endomorphism algebra \(k[t]/(t^r)\), and both
\(\operatorname{Ext}^i_A(P,E_r)\) and \(\operatorname{Ext}^i_A(E_r,P)\)
have dimension \(\binom{g}{i}\) for \(0\leq i\leq g\).
Theorem~\ref{thm:stable-acm-semihomogeneous} gives slope-stable ACM
bundles from simple semihomogeneous bundles. Their ranks are determined
by the polarization type and the rational slope class.

We also relate the ACM property to the cohomological support loci
studied by Pareschi and Popa \cite{PP1,PP3}. When \(F\) and
\(F^\vee\otimes L\) are \(M\)-regular, Proposition~\ref{prop:Mregular-ACM} shows that \(F\)
is ACM exactly when the intermediate cohomology of both bundles
vanishes. Theorem~\ref{thm:GV-ACM} shows that if \(E\) and \(E^\vee\) are GV, then
for a general \(\alpha\in\widehat A\) the bundle \(E\otimes P_\alpha\)
is ACM with respect to every ample line bundle \(L\). We also give
examples of ACM bundles that fail \(M\)-regularity.

\section{Preliminaries}
\label{sec:preliminaries}

Throughout, \(k\) is an algebraically closed field. Unless otherwise specified, we assume \(\operatorname{char}k=0\). All varieties are integral and projective over \(k\).

\begin{definition}
An abelian variety over \(k\) is a complete connected algebraic group variety over \(k\). If \(A\) is an abelian variety, its dual abelian variety is
\[
\widehat A=\operatorname{Pic}^{0}(A).
\]
A polarization on \(A\) is the numerical class of an ample line bundle. In this article, by a polarized abelian variety we mean a pair \((A,L)\), where \(L\) is an ample line bundle on \(A\).
\end{definition}

For \(a\in A(k)\), let \(t_a:A\to A\) denote translation by \(a\).

\begin{definition}
The subgroup \(\operatorname{Pic}^{0}(A)\subseteq \Pic(A)\) consists of line bundles algebraically equivalent to \(\OO_A\). Equivalently, it is the connected component of the identity of the Picard scheme. Every element of \(\operatorname{Pic}^{0}(A)\) is algebraically trivial.
\end{definition}
\begin{remark}
On $A$ the group $\Pic^0(A)$, the connected component of the identity in
$\Pic(A)$, equals the group of algebraically trivial line bundles. Since
algebraic equivalence implies numerical equivalence, every such bundle is numerically trivial. On an abelian variety the converse also holds: $\operatorname{NS}(A)$ is torsion-free, so algebraic and numerical equivalence of line bundles coincide.
\end{remark}

\begin{lemma}
\label{lem:pic0-numerically-trivial}
If \(P\in \operatorname{Pic}^{0}(A)\), then \(c_1(P)=0\) in \(\mathrm{NS}(A)\). Consequently,

\[
c_1(P\otimes M)=c_1(M)
\]
for every line bundle \(M\) on \(A\).
\end{lemma}

\begin{proof}
By definition, \(\operatorname{Pic}^{0}(A)\) is the subgroup of line bundles algebraically equivalent to zero. Algebraic equivalence implies numerical equivalence. Therefore the class of \(P\) in the Néron--Severi group

\[
\mathrm{NS}(A)=\Pic(A)/\operatorname{Pic}^{0}(A)
\]
is zero. Hence \(c_1(P)=0\) in \(\mathrm{NS}(A)\). The formula for \(c_1(P\otimes M)\) follows from additivity of the first Chern class.
\end{proof}

\begin{lemma}
\label{lem:pic0-preserves-ampleness}
Let \(L\) be ample on \(A\), and let \(P\in \operatorname{Pic}^{0}(A)\). Then \(L\otimes P\) is ample.
\end{lemma}

\begin{proof}
By Lemma \ref{lem:pic0-numerically-trivial}, \(L\otimes P\) is numerically equivalent to \(L\). Ampleness of line bundles is invariant under numerical equivalence on a projective variety by Kleiman's criterion. Since \(L\) is ample, \(L\otimes P\) is ample.
\end{proof}

\begin{lemma}
\label{lem:canonical-trivial}
Let \(A\) be an abelian variety of dimension \(g\). Then

\[
\omega_A\simeq \OO_A.
\]
\end{lemma}

\begin{proof}
The tangent bundle \(T_A\) is trivial. Indeed, every tangent vector at the identity element \(0\in A\) extends uniquely to a translation-invariant vector field on \(A\). Hence

\[
T_A\simeq T_{A,0}\otimes_k \OO_A.
\]
Taking duals and determinants gives

\[
\omega_A=\det\Omega_A^1\simeq \det(T_{A,0}^{\vee})\otimes_k \OO_A\simeq \OO_A.
\]
\end{proof}

\begin{corollary}[Serre duality on an abelian variety]
\label{cor:serre-duality}
For every coherent sheaf \(\mathcal F\) on \(A\), there is a functorial duality

\[
H^i(A,\mathcal F)^\vee \simeq \Ext^{g-i}_A(\mathcal F,\OO_A).
\]
In particular, if \(M\) is a line bundle on \(A\), then
\[
H^i(A,M)^\vee\simeq H^{g-i}(A,M^{-1}).
\]
\end{corollary}

\begin{proof}
$A$ is smooth and projective, hence Cohen--Macaulay, and its dualizing sheaf is
$\omega_A\cong\mathcal O_A$ by Lemma~\ref{lem:canonical-trivial}; apply \cite[Chapter~III, Theorem~7.6]{Hartshorne}(see also  \cite[Chapter~III, Corollary~7.7]{Hartshorne}).
\end{proof}
\begin{theorem}[Riemann--Roch] \label{thm:riemann_roch}
Let $A$ be an abelian variety of dimension $g$, and let $M$ be a line bundle on $A$. Then
\[
 \chi(A, M) = \frac{c_1(M)^g}{g!},
\]
where the right-hand side denotes the degree of the zero-cycle $\int_A c_1(M)^g$. 
In particular, if $L$ is ample and $P \in \operatorname{Pic}^{0}(A)$, then
\[
    \chi(A, P \otimes L^{\otimes m}) = m^g \frac{(L^g)}{g!}
\]
for every $m \in \mathbb{Z}$.
\end{theorem}

\begin{proof}
Follows from the Hirzebruch–Riemann–Roch theorem and Lemma~\ref{lem:canonical-trivial}.
\end{proof}

\begin{definition}\label{ACM}
Let \(A\) be a smooth projective variety of dimension \(n\), and let \(H\) be an ample line bundle on \(A\). A vector bundle \(\mathcal F\) on \(A\) is called arithmetically Cohen--Macaulay, or ACM, with respect to \(H\), if
\[
H^i(A,\mathcal F\otimes H^{\otimes m})=0
\]
for every \(m\in\mathbb Z\) and every \(0<i<n\).
\end{definition}

\begin{remark}
    When $L$ is very ample, an ACM bundle $E$ corresponds to the maximal Cohen--Macaulay
module $\bigoplus_{m}H^0(A,E\otimes L^{\otimes m})$; note, however, that an abelian
variety of dimension $g\ge2$ is never arithmetically Cohen--Macaulay, since
$H^1(A,\mathcal O_A)\neq0$. In this paper $L$ is only assumed to be ample, and
``ACM'' is understood in the cohomological sense of Definition~\ref{ACM}.
\end{remark}

\begin{remark}
If \(n=1\), the ACM condition is vacuous, since there is no integer \(i\) satisfying \(0<i<1\).
\end{remark}

\begin{definition}
Let \(A\) be a smooth projective variety of dimension \(n\), and let \(H\) be an ample line bundle. A vector bundle \(\mathcal F\) on \(A\) is called Ulrich with respect to \(H\) if

\[
H^i(A,\mathcal F\otimes H^{-j})=0
\]
for all \(i\) and all \(1\leq j\leq n\).
\end{definition}

\begin{remark}
Ulrich (and ACM) bundles are usually defined with respect to a very ample line
bundle $\mathcal O(1)$; here we allow an arbitrary ample $H$, and the definitions are
purely cohomological.
\end{remark}

\begin{lemma}[Base change]\label{lem:basechange}
For every $F\in D^b(A)$ and every closed point $\beta\in\widehat A$,
\[
\Phi(F)\otimes^{\mathbf L}_{\mathcal O_{\widehat A}}k(\beta)\;\simeq\;
R\Gamma(A,F\otimes P_\beta).
\]
\end{lemma}
\begin{proof}
Cohomology and base change for the proper morphism $p_{\widehat A}$ and the complex $p_A^*F\otimes\mathcal P$; see \cite[\S5]{MumfordAV}.
\end{proof}

\section{Algebraically Trivial ACM Line Bundles}
\label{sec:proof-main}

The rank-one case follows from the Fourier--Mukai acyclicity of
nontrivial elements of $\operatorname{Pic}^{0}(A)$. Together with
Kodaira vanishing and Serre duality, this gives the ACM property with
respect to every ample polarization.

\begin{theorem}[Acyclicity of nontrivial algebraically trivial line
bundles]\label{thm:pic0-acyclicity}
Let $A$ be an abelian variety of dimension $g$, and let
\[
P\in\operatorname{Pic}^{0}(A)\setminus\{\mathcal{O}_{A}\}.
\]
Then
\[
H^{i}(A,P)=0
\qquad
\text{for every }0\leq i\leq g.
\]
Equivalently,
\[
\mathbf{R}\Gamma(A,P)\simeq 0.
\]
\end{theorem}

We give a Fourier–Mukai proof. Let $\mathcal P$ be the normalized Poincaré line bundle on \(A\times \widehat A\). Let
\[
p_A:A\times \widehat A\to A,\qquad p_{\widehat A}:A\times \widehat A\to \widehat A
\]
be the projections. The Fourier--Mukai functor is

\[
\Phi_{\mathcal P}:\mathbf D^b(\operatorname{Coh}A)\to \mathbf D^b(\operatorname{Coh}\widehat A),
\]
defined by
\[
\Phi_{\mathcal P}(\mathcal F)
=
\textbf{R}p_{\widehat A,*}(p_A^*\mathcal F\otimes \mathcal P).
\]
We use the following standard facts due to Mukai \cite{MukaiDuality}; see also \cite[Chapter 9]{HuybrechtsFM}.

\begin{theorem}[Mukai]
\label{thm:mukai-standard}
The functor \(\Phi_{\mathcal P}\) is an equivalence of triangulated categories. Moreover,
\[
\Phi_{\mathcal P}(\OO_A)\simeq k(0_{\widehat A})[-g],
\]
where \(k(0_{\widehat A})\) denotes the skyscraper sheaf at the origin of \(\widehat A\).
\end{theorem}

\begin{lemma}\label{lem:tensor-translation-FM}
For $\alpha\in\widehat A(k)$, put
\[
P_\alpha:=\mathcal P|_{A\times\{\alpha\}},
\qquad
t_\alpha:\widehat A\longrightarrow\widehat A,
\quad
\beta\longmapsto\beta+\alpha.
\]
Then, for every $F\in D^b(\operatorname{Coh}A)$,
\[
\Phi_{\mathcal P}(F\otimes P_\alpha)
\simeq
t_\alpha^*\Phi_{\mathcal P}(F).
\]
In particular,
\[
\Phi_{\mathcal P}(P_\alpha)
\simeq k(-\alpha)[-g].
\]
\end{lemma}

\begin{proof}
Set $f_\alpha:=1_A\times t_\alpha$. We first show that
\begin{equation}
\label{eq:poincare-translation}
f_\alpha^*\mathcal P
\simeq
p_A^*P_\alpha\otimes\mathcal P.
\end{equation}
For every $\beta\in\widehat A(k)$, the restrictions of
both sides to $A\times\{\beta\}$ are isomorphic to
$P_{\alpha+\beta}$, since
\[
P_\alpha\otimes P_\beta\simeq P_{\alpha+\beta}.
\]
Consequently, their quotient
\[
\mathcal Q
:=
f_\alpha^*\mathcal P
\otimes
(p_A^*P_\alpha\otimes\mathcal P)^{-1}
\]
restricts trivially to every fibre of $p_{\widehat A}$. By the seesaw principle,
$\mathcal Q\simeq p_{\widehat A}^*M$ for a line bundle $M$ on $\widehat A$.
The normalization of $\mathcal P$ implies that
$\mathcal Q|_{\{0_A\}\times\widehat A}$ is trivial.
Restricting $p_{\widehat A}^*M$ to this section
therefore gives $M\simeq\mathcal O_{\widehat A}$,
proving \eqref{eq:poincare-translation}.

Since $p_A\circ f_\alpha=p_A$, we obtain
\[
p_A^*(F\otimes P_\alpha)\otimes\mathcal P
\simeq
f_\alpha^*(p_A^*F\otimes\mathcal P).
\]
The square
\[
\begin{array}{ccc}
A\times\widehat A
& \xrightarrow{\ f_\alpha\ } &
A\times\widehat A\\
{\scriptstyle p_{\widehat A}}\downarrow
&&
\downarrow{\scriptstyle p_{\widehat A}}\\
\widehat A
& \xrightarrow{\ t_\alpha\ } &
\widehat A
\end{array}
\]
is Cartesian, and $t_\alpha$ is an isomorphism,
hence flat. Flat base change gives
\begin{align*}
\Phi_{\mathcal P}(F\otimes P_\alpha)
&=
Rp_{\widehat A,*}
\bigl(p_A^*(F\otimes P_\alpha)\otimes\mathcal P\bigr)\\
&\simeq
Rp_{\widehat A,*}
f_\alpha^*(p_A^*F\otimes\mathcal P)\\
&\simeq
t_\alpha^*
Rp_{\widehat A,*}(p_A^*F\otimes\mathcal P)\\
&=
t_\alpha^*\Phi_{\mathcal P}(F).
\end{align*}

Taking $F=\mathcal O_A$ and applying Theorem~\ref{thm:mukai-standard},
we conclude that
\[
\Phi_{\mathcal P}(P_\alpha)
\simeq
t_\alpha^*k(0_{\widehat A})[-g]
\simeq
k(-\alpha)[-g],
\]
because $t_\alpha^{-1}(0_{\widehat A})=-\alpha$.
\end{proof}

\begin{proof}[Proof of Theorem \ref{thm:pic0-acyclicity}]
Let \(P=P_\alpha\), where \(\alpha\in\widehat A(k)\). Since \(P\neq \OO_A\), we have \(\alpha\neq 0_{\widehat A}\).

By definition of the Fourier--Mukai transform, the derived fiber of \(\Phi_{\mathcal P}(P)\) at a point \(\beta\in\widehat A\) is

\[
\Phi_{\mathcal P}(P)\otimes^{\mathbf L} k(\beta)
\simeq
\textbf{R}\Gamma(A,P\otimes P_\beta).
\]
Taking \(\beta=0_{\widehat A}\), we obtain

\[
\Phi_{\mathcal P}(P)\otimes^{\mathbf L} k(0_{\widehat A})
\simeq
\textbf{R}\Gamma(A,P).
\]
But by Lemma \ref{lem:tensor-translation-FM},

\[
\Phi_{\mathcal P}(P)\simeq k(-\alpha)[-g].
\]
Since \(\alpha\neq 0\), the skyscraper sheaf \(k(-\alpha)\) has zero derived fiber at \(0_{\widehat A}\). Therefore

\[
\textbf{R}\Gamma(A,P)=0.
\]
Hence

\[
H^i(A,P)=0
\]
for all \(i\).
\end{proof}


We now prove the rank-one case. 

\begin{theorem}\label{thm:main}
Let $(A,L)$ be a polarized abelian variety of dimension $g\geq 1$
over an algebraically closed field of characteristic zero. Then every
nontrivial line bundle

\[
P\in \operatorname{Pic}^{0}(A)
\]
is arithmetically Cohen--Macaulay with respect to $L$.
\end{theorem}

\begin{proof}
Let $m\in \mathbb{Z}$. If $m>0$, then $P\otimes L^{\otimes m}$ is
ample. Since $\omega_A\simeq \mathcal{O}_A$, Kodaira vanishing gives

\[
H^i\bigl(A,P\otimes L^{\otimes m}\bigr)=0
\qquad \text{for all } i>0.
\]

If $m=0$, Theorem~\ref{thm:pic0-acyclicity} gives

\[
H^i(A,P)=0
\qquad \text{for all } i.
\]

If $m<0$, Serre duality gives

\[
H^i\bigl(A,P\otimes L^{\otimes m}\bigr)^\vee
\simeq
H^{g-i}\bigl(A,P^{-1}\otimes L^{\otimes(-m)}\bigr).
\]
The line bundle on the right is ample, so Kodaira vanishing shows that
this group is zero whenever $g-i>0$, equivalently whenever $i<g$.

Therefore, for every $m\in\mathbb{Z}$,

\[
H^i\bigl(A,P\otimes L^{\otimes m}\bigr)=0
\qquad \text{for } 0<i<g.
\]
Hence $P$ is ACM with respect to $L$.
\end{proof}
\begin{remark}\label{chf}
The rank-one ACM statement also holds over an algebraically
closed field of arbitrary characteristic. In positive
characteristic, the proof uses the vanishing theorem for
ample line bundles on abelian varieties in place of
Kodaira vanishing, see \cite[\S16]{MumfordAV} .
\end{remark}
\begin{remark}\label{rem:ballico}
 On an abelian surface $X$ with $\mathrm{Num}(X)\cong\mathbb{Z}$, Ballico 
\cite{Ballico2008} observed that every nontrivial 
$P\in\operatorname{Pic}^0(X)$ is totally arithmetically Cohen--Macaulay.
Ballico used such line bundles to classify rank-two totally ACM bundles 
on these surfaces: they are exactly the extensions of two numerically 
trivial line bundles. 
\end{remark}
\begin{definition}
The \emph{ACM complexity} of $(A,L)$ is
\[
c_{\mathrm{ACM}}(A,L)
=
\min\left\{
\operatorname{rk}(E)
\;\middle|\;
E\text{ is a nonzero ACM vector bundle on }(A,L)
\right\}.
\]
If no such bundle exists, we set
$c_{\mathrm{ACM}}(A,L)=\infty$.
\end{definition}

\begin{corollary}
\label{cor:acm-complexity}
For every polarized abelian variety $(A,L)$,
\[
c_{\mathrm{ACM}}(A,L)=1.
\]
\end{corollary}
\begin{proof}
A nontrivial element of $\operatorname{Pic}^{0}(A)$ is an ACM bundle
of rank one by Theorem~\ref{thm:main}. Since every nonzero
vector bundle has positive rank, the minimum is one.
\end{proof}
\subsection{Cohomology and Consequences}
\label{sec:cohomology-consequences}

Let $(A,L)$ be a polarized abelian variety of dimension $g$, let

\[
P\in\operatorname{Pic}^{0}(A)\setminus\{\mathcal{O}_A\},
\qquad
h=\chi(A,L)=\frac{L^g}{g!}.
\]
The vanishing theorem of Section~\ref{sec:proof-main}, together with
the Hirzebruch–Riemann–Roch theorem, determines the complete cohomology table of
$P\otimes L^{\otimes m}$.

\begin{theorem}[Cohomology table]
\label{thm:cohomology-table}
For every $m\in\mathbb{Z}$ and every $0\leq i\leq g$,

\[
h^i\bigl(A,P\otimes L^{\otimes m}\bigr)
=
\begin{cases}
m^g h, & m>0 \text{ and } i=0,\\[2mm]
(-m)^g h, & m<0 \text{ and } i=g,\\[2mm]
0, & \text{otherwise}.
\end{cases}
\]
\end{theorem}

\begin{proof}
For $m>0$, Kodaira vanishing gives

\[
H^i\bigl(A,P\otimes L^{\otimes m}\bigr)=0
\qquad (i>0).
\]
For $m=0$, Theorem~\ref{thm:pic0-acyclicity} gives

\[
H^i(A,P)=0
\qquad\text{for all }i.
\]
For $m<0$, Serre duality and Kodaira vanishing give

\[
H^i\bigl(A,P\otimes L^{\otimes m}\bigr)=0
\qquad (i<g).
\]
Finally, Riemann--Roch gives

\[
\chi\bigl(A,P\otimes L^{\otimes m}\bigr)=m^g h.
\]
The asserted dimensions now follow from the vanishings above.
\end{proof}

We record several consequences of the cohomology computation.

\begin{corollary}
\label{cor:structure-not-acm}
If $g\geq2$, then $\mathcal{O}_A$ is not ACM with respect to any
polarization on $A$.
\end{corollary}

\begin{proof}
Since

\[
h^1(A,\mathcal{O}_A)=g>0,
\]
the required intermediate cohomology vanishing fails for the
untwisted bundle $\mathcal{O}_A$.
\end{proof}

\begin{corollary}
\label{cor:pic0-classification}
Assume $g\geq2$. For $Q\in\operatorname{Pic}^{0}(A)$, the following
conditions are equivalent:

\[
Q\text{ is ACM with respect to }L
\qquad\Longleftrightarrow\qquad
Q\not\simeq\mathcal{O}_A.
\]
Thus $\mathcal{O}_A$ is the unique non-ACM element of
$\operatorname{Pic}^{0}(A)$.
\end{corollary}

\begin{proof}
The nontrivial elements are ACM by
Theorem~\ref{thm:main}, whereas $\mathcal{O}_A$ is not ACM
by Corollary~\ref{cor:structure-not-acm}.
\end{proof}

\begin{corollary}
\label{cor:rank-one-family}
The punctured dual abelian variety

\[
\operatorname{Pic}^{0}(A)\setminus\{\mathcal{O}_A\}
\]
parametrizes a $g$-dimensional family of pairwise non-isomorphic
rank-one ACM bundles.
\end{corollary}

\begin{proof}
The dual abelian variety has dimension $g$, and removing its origin
does not change its dimension. The assertion follows from
Theorem~\ref{thm:main}.
\end{proof}

\begin{corollary}
\label{cor:stable-rank-one}
Every nontrivial $P\in\operatorname{Pic}^{0}(A)$ is a slope-stable
ACM line bundle with respect to $L$.
\end{corollary}

\begin{proof}
Every line bundle is slope-stable, and the ACM property follows from
Theorem~\ref{thm:main}.
\end{proof}

We next determine which twists of these line bundles are Ulrich.

\begin{theorem}
\label{thm:ulrich-twists}
Let

\[
M_a=P\otimes L^{\otimes a},
\qquad a\in\mathbb{Z}.
\]
Then:
\begin{enumerate}
    \item if $g=1$, the bundle $M_a$ is Ulrich if and only if $a=1$;
    \item if $g\geq2$, the bundle $M_a$ is never Ulrich.
\end{enumerate}
\end{theorem}

\begin{proof}
The bundle $M_a$ is Ulrich if and only if

\[
H^i\bigl(A,M_a\otimes L^{-j}\bigr)=0
\]
for every $i$ and every $1\leq j\leq g$. Since

\[
M_a\otimes L^{-j}
\simeq
P\otimes L^{\otimes(a-j)},
\]
Theorem~\ref{thm:cohomology-table} shows that this bundle is acyclic
if and only if $a=j$. When $g=1$, this gives $a=1$. When $g\geq2$,
the equalities $a=j$ cannot hold simultaneously for every
$1\leq j\leq g$.
\end{proof}

\section{Indecomposable ACM Bundles of Arbitrary Rank}
 Throughout this section $(A,L)$ is a polarized abelian variety of dimension $g\ge2$
over an algebraically closed field $k$ of characteristic zero, and
$P\in\operatorname{Pic}^0(A)\setminus\{\mathcal O_A\}$. In this section we
construct indecomposable ACM bundles of every rank $r\ge2$ by iterated non-split
extensions.

\begin{definition}[Iterated Extension Sequence]\label{def:iter}
Define a sequence of vector bundles $E_r$ (for $r \ge 1$) inductively as follows:
\begin{itemize}
    \item \textbf{Base case:} $E_1 := P$.
    \item \textbf{Inductive step:} For $r \ge 2$, given $E_{r-1}$, choose a nonzero extension class $\xi_r \in \Ext^1_A(P, E_{r-1}) \setminus \{0\}$ (whose existence is guaranteed by Lemma \ref{lem:existence} below). Let $E_r$ be the middle term of the corresponding non-split short exact sequence:
    \begin{equation}\label{eq:ses}
        0 \longrightarrow E_{r-1} \xrightarrow{\iota_r} E_r \xrightarrow{\pi_r} P \longrightarrow 0.
    \end{equation}
\end{itemize}
\end{definition}

\begin{lemma}[Chern Data of $E_r$]\label{lem:chern}
For every $r \ge 1$, $\rk(E_r) = r$ and $c_1(E_r) = 0$ in $NS(A)$.
\end{lemma}
\begin{proof}
We proceed by induction on $r$. 

\textbf{Base case ($r=1$):} By definition, $E_1 = P$. Since $P$ is a line bundle, $\rk(E_1) = 1$. Because $P \in \operatorname{Pic}^{0}(A)$, its first Chern class is numerically trivial, so $c_1(E_1) = 0$.

\textbf{Inductive step:} Assume $\rk(E_{r-1}) = r-1$ and $c_1(E_{r-1}) = 0$. In any short exact sequence of locally free sheaves, rank and first Chern class are additive. Applying this to \eqref{eq:ses} yields:
\begin{align*}
    \rk(E_r) &= \rk(E_{r-1}) + \rk(P) = (r-1) + 1 = r, \\
    c_1(E_r) &= c_1(E_{r-1}) + c_1(P) = 0 + 0 = 0.
\end{align*}
This completes the induction.
\end{proof}

\begin{lemma}[Existence of Non-split Inductive Extensions] \label{lem:existence}
For every $r\ge 2$ and every $P\in\operatorname{Pic}^0(A)\setminus\{\mathcal O_A\}$,
\[
  \operatorname{Ext}^1_A(P,E_{r-1})\;\cong\;H^1\bigl(A,\,P^{\vee}\otimes E_{r-1}\bigr)\;\neq\;0 .
\]
\end{lemma}

\begin{proof}
The proof proceeds in six steps.

\medskip
\noindent\textbf{Step~1: The $\Ext$--$H^1$ isomorphism.}

Since $P$ is locally free of rank one, all higher local $\mathcal{E}xt$
sheaves vanish:
\[
  \mathcal{E}xt^q_{\OA}(P, E_{r-1}) = 0 \quad \text{for all } q > 0,
\]
and $\mathcal{H}om(P, E_{r-1}) \cong P^\vee \otimes E_{r-1}$. The local-to-global Ext spectral sequence therefore gives canonical
isomorphisms
\[
  \Ext^i_A(P, E_{r-1})
  \;\cong\;
  H^i(A,\, P^\vee \otimes E_{r-1})
  \quad \text{for all } i \geq 0.
\]
In particular,
\[
  \Ext^1_A(P, E_{r-1})
  \;\cong\;
  H^1(A,\, P^\vee \otimes E_{r-1}).
\]
It suffices to prove $H^1(A, P^\vee \otimes E_{r-1}) \neq 0$.

\medskip
\noindent\textbf{Step~2: The unipotent bundle $U_s$.}

Define $U_s := P^\vee \otimes E_s$ for each $s \geq 1$.
We claim that $U_s$ is a \emph{unipotent vector bundle of rank $s$},
i.e.\ it admits a filtration
\[
  0 = U_s^{(0)} \subset U_s^{(1)} \subset \cdots
  \subset U_s^{(s)} = U_s
\]
with each successive quotient isomorphic to $\OA$.

We verify this by induction on $s$.
\begin{itemize}
  \item \emph{Base case} ($s=1$): $U_1 = P^\vee \otimes P \cong \OA$.
  \item \emph{Inductive step} ($s \geq 2$): Tensoring \eqref{eq:ses}
        (for index $s$) with the locally free sheaf $P^\vee$ preserves
        exactness, giving
        \begin{equation}\label{eq:Us-seq}
          0 \;\to\; U_{s-1} \;\to\; U_s \;\to\; \OA \;\to\; 0.
        \end{equation}
        By induction $U_{s-1}$ is unipotent of rank $s-1$, so $U_s$
        is unipotent of rank $s$.
\end{itemize}

\medskip
\noindent\textbf{Step~3: Fourier--Mukai transform of $U_s$.}

\begin{claim}\label{claim:FM-Us}
Let $U_s := P^{\vee} \otimes E_s$ for each $s \ge 1$. Then there exists an isomorphism
\[
\Phi_{\mathcal{P}}(U_s) \simeq \FF_s[-g]
\]
in $D^b(\operatorname{Coh} \widehat{A})$, where $\FF_s$ is a non-zero coherent sheaf on $\widehat{A}$ of finite length $s$ whose set-theoretic support is $\{0_{\widehat{A}}\}$.
\end{claim}

\begin{proof}
We argue by induction on $s$.

\paragraph{Base case ($s=1$).}
$U_1 = P^{\vee} \otimes P \cong \mathcal{O}_A$. The Fourier--Mukai transform of the structure sheaf of an abelian variety is the skyscraper sheaf at the origin of the dual (with a cohomological shift):
\[
\Phi_{\mathcal{P}}(\mathcal{O}_A) \simeq k(0_{\widehat{A}})[-g].
\]
Thus $\FF_1 := k(0_{\widehat{A}})$ has length $1$ and support $\{0_{\widehat{A}}\}$.

\paragraph{Inductive step ($s \ge 2$).}
Tensoring the exact sequence $0 \to E_{s-1} \to E_s \to P \to 0$ with the flat sheaf $P^{\vee}$ yields
\[
0 \longrightarrow U_{s-1} \longrightarrow U_s \longrightarrow \mathcal{O}_A \longrightarrow 0.
\]
Applying the exact functor $\Phi_{\mathcal{P}}$ gives a distinguished triangle in $D^b(\widehat{A})$:
\[
\Phi_{\mathcal{P}}(U_{s-1}) \longrightarrow \Phi_{\mathcal{P}}(U_s) \longrightarrow \Phi_{\mathcal{P}}(\mathcal{O}_A) \longrightarrow \Phi_{\mathcal{P}}(U_{s-1})[1].
\]
By the induction hypothesis, this is
\begin{equation}\label{long}
    \FF_{s-1}[-g] \longrightarrow B \longrightarrow k(0_{\widehat{A}})[-g] \longrightarrow \FF_{s-1}[-g+1]
\end{equation}
where $B := \Phi_{\mathcal{P}}(U_s)$.

\subparagraph{Vanishing of $H^q(B)$ for $q \neq g$.}
Since $\FF_{s-1}[-g]$ and $k(0_{\widehat{A}})[-g]$ are concentrated in cohomological degree $g$, the long exact sequence of cohomology sheaves associated to \eqref{long} has the following shape:
\[
\cdots \to 0 \to H^q(B) \to 0 \to \cdots \quad\text{for } q \le g-1 \text{ and } q \ge g+2,
\]
which forces $H^q(B)=0$ in those ranges. At $q=g+1$ the relevant segment is
\[
k(0_{\widehat{A}}) \to 0 \to H^{g+1}(B) \to 0,
\]
and exactness at $H^{g+1}(B)$ gives $H^{g+1}(B)=0$. At $q=g$ we obtain the short exact sequence
\begin{equation}\label{length}
    0 \to \FF_{s-1} \to H^g(B) \to k(0_{\widehat{A}}) \to 0.
\end{equation}

Thus $H^q(B)=0$ for all $q \neq g$.

\subparagraph{Identification of $B$.}
Since $B$ has cohomology only in degree $g$, it is quasi-isomorphic to its sole cohomology sheaf placed in that degree:
\[
B \simeq H^g(B)[-g].
\]
Define $\FF_s := H^g(B)$. Then $B \simeq \FF_s[-g]$, i.e.
\[
\Phi_{\mathcal{P}}(U_s) \simeq F_s[-g].
\]

\subparagraph{Length and support.}
From the short exact sequence \eqref{length} we obtain
\[
0 \longrightarrow \FF_{s-1} \longrightarrow \FF_s \longrightarrow k(0_{\widehat{A}}) \longrightarrow 0.
\]
Additivity of length gives
\[
\ell(\FF_s) = \ell(\FF_{s-1}) + \ell\bigl(k(0_{\widehat{A}})\bigr) = (s-1) + 1 = s.
\]
Moreover,
\[
\operatorname{Supp}(\FF_s) \subseteq \operatorname{Supp}(\FF_{s-1}) \cup \operatorname{Supp}\bigl(k(0_{\widehat{A}})\bigr) = \{0_{\widehat{A}}\}.
\]
Since $s \ge 1$, we have $\FF_s \neq 0$, and therefore $\operatorname{Supp}(\FF_s) = \{0_{\widehat{A}}\}$.

This completes the induction.
\end{proof}

\medskip
\noindent\textbf{Step~4: Cohomology via the derived fiber formula.}

The derived fiber of $\Phi_{\mathcal P}(\mathcal{G})$ at a closed point
$\beta \in \wA$ is
\[
  \Phi_{\mathcal P}(\mathcal{G}) \otimes^{\mathbf{L}}_{\mathcal{O}_{\wA}}
  k(\beta) \;\simeq\; \textbf{R}\Gamma(A,\, \mathcal{G} \otimes P_\beta).
\]
Taking $\beta = 0_{\wA}$ and using $P_0 \cong \OA$:
\[
  \textbf{R}\Gamma(A,\, \mathcal{G})
  \;\simeq\;
  \Phi_{\mathcal P}(\mathcal{G})
  \otimes^{\mathbf{L}}_{\mathcal{O}_{\wA}} k(0_{\wA}).
\]
Applying this to $\mathcal{G} = U_{r-1}$ and substituting
Claim~\ref{claim:FM-Us}:
\[
  \textbf{R}\Gamma(A,\, U_{r-1})
  \;\simeq\;
  \FF_{r-1}[-g]
  \otimes^{\mathbf{L}}_{\mathcal{O}_{\wA}} k(0_{\wA}).
\]
Let $R := \mathcal{O}_{\wA, 0}$ be the local ring at $0_{\wA}$ and
$\kappa := R/\mathfrak{m}_0 \cong k$ its residue field.
Since $\FF_{r-1}$ has finite-length support at $0_{\wA}$, it
corresponds to a finite-length $R$-module, and the derived tensor
product is computed entirely over $R$:
\[
  \textbf{R}\Gamma(A,\, U_{r-1})
  \;\simeq\;
  \FF_{r-1} \otimes^{\mathbf{L}}_R \kappa\,[-g].
\]
Passing to cohomology:
\begin{equation}\label{eq:Hi-Tor}
  H^i(A,\, U_{r-1})
  \;\cong\;
  \Tor^R_{g-i}(\FF_{r-1},\, \kappa)
  \quad \text{for all } i.
\end{equation}
Setting $i = 1$:
\[
  H^1(A,\, P^\vee \otimes E_{r-1})
  \;=\;
  H^1(A,\, U_{r-1})
  \;\cong\;
  \Tor^R_{g-1}(\FF_{r-1},\, \kappa).
\]

\medskip
\noindent\textbf{Step~5: Nonvanishing via the Auslander--Buchsbaum
formula and the minimal free resolution.}

Since $\wA$ is a smooth abelian variety of dimension~$g$, the local
ring $R = \mathcal{O}_{\wA, 0}$ is a \emph{regular local ring of
dimension~$g$}. The module $\FF_{r-1}$ is a nonzero finitely generated
$R$-module of finite length $\ell(\FF_{r-1}) = r - 1 \geq 1$.
A nonzero finite-length module over a positive-dimensional local ring
has depth zero:
\[
  \depth_R(\FF_{r-1}) = 0.
\]
Since $R$ is regular local (hence Cohen--Macaulay), every finitely
generated $R$-module has finite projective dimension, so the
Auslander-Buchsbaum formula applies:
\[
  \pd_R(\FF_{r-1}) + \depth_R(\FF_{r-1}) = \depth(R) = g,
\]
giving $\pd_R(\FF_{r-1}) = g$.

Let
\[
  0 \;\to\; R^{\beta_g}
    \;\xrightarrow{d_g}\; R^{\beta_{g-1}}
    \;\to\; \cdots
    \;\to\; R^{\beta_1}
    \;\xrightarrow{d_1}\; R^{\beta_0}
    \;\to\; \FF_{r-1}
    \;\to\; 0
\]
be the \emph{minimal free resolution} of $\FF_{r-1}$ over $R$.
Minimality means every differential $d_j$ has its matrix entries
in $\mathfrak{m}_0$. In particular, since $\pd_R(\FF_{r-1}) = g$,
we have
\[
  \beta_g \;\neq\; 0.
\]

Exactness of the resolution at the leftmost term $R^{\beta_g}$ requires the map $d_g \colon R^{\beta_g} \to R^{\beta_{g-1}}$ to be injective. Because the resolution is minimal, the image of $d_g$ is contained in $\mathfrak{m}_0 R^{\beta_{g-1}}$. If $\beta_{g-1} = 0$, then $R^{\beta_{g-1}} = 0$, which forces $d_g$ to be the zero map. An injective map from the non-zero free module $R^{\beta_g}$ to zero is impossible, which contradicts $\beta_g \ne 0$. Therefore:
$$ \beta_{g-1} \ne 0. $$

Since the resolution is minimal, tensoring with $\kappa = R/\mathfrak{m}_0$
kills all differentials, yielding
\[
  \Tor^R_j(\FF_{r-1},\, \kappa)
  \;\cong\;
  \kappa^{\beta_j}
  \quad \text{for all } 0 \leq j \leq g.
\]
In particular,
\[
  \Tor^R_{g-1}(\FF_{r-1},\, \kappa)
  \;\cong\;
  \kappa^{\beta_{g-1}}
  \;\neq\; 0.
\]

\medskip
\noindent\textbf{Step~6: Conclusion.}

Combining the results of Steps 1--5:
\[
  \Ext^1_A(P,\, E_{r-1})
  \;\cong\;
  H^1(A,\, P^\vee \otimes E_{r-1})
  \;=\;
  H^1(A,\, U_{r-1})
  \;\cong\;
  \Tor^R_{g-1}(\FF_{r-1},\, \kappa)
  \;\neq\; 0.
\]
Hence there exists a nonzero class $\xi_r \in \Ext^1_A(P, E_{r-1})
\setminus \{0\}$, and the corresponding non-split extension
\[
  0 \;\to\; E_{r-1} \;\to\; E_r \;\to\; P \;\to\; 0
\]
exists at every inductive step $r \geq 2$.
\end{proof}

\begin{proposition}[Global sections of $\mathcal{H}om(P, E_r)$]
\label{prop:hom-Er-P}
For every integer $r\ge 1$, set $U_r:=P^{\vee}\otimes E_r\cong\mathcal{H}om(P,E_r)$. Then
\[
  h^0(A,U_r)=1 .
\]
\end{proposition}

\begin{proof}
We proceed by induction on $r$.

\smallskip\noindent\textit{Base case} $(r = 1)$.
By definition $E_1 = P$, so $U_1 = P^{\vee} \otimes P \cong \mathcal{O}_A$.
Since $A$ is a proper variety over an algebraically closed field,
$h^0(A, \mathcal{O}_A) = 1$.

\smallskip\noindent\textit{Inductive step} $(r \geq 2)$.
Assume $h^0(A, U_{r-1}) = 1$.  The defining non-split extension
\[
  0 \longrightarrow E_{r-1} \xrightarrow{\iota_r} E_r \xrightarrow{\pi_r}
  P \longrightarrow 0
\]
yields, after tensoring with $P^{\vee}$, the exact sequence of unipotent bundles
(see Lemma~\ref{lem:existence}, Step 2)
\begin{equation}\label{seq:Us}
  0 \longrightarrow U_{r-1} \longrightarrow U_r \longrightarrow \mathcal{O}_A
  \longrightarrow 0.
\end{equation}
We apply the left-exact functor $\mathrm{Hom}_A(P, -)$ to the sequence
$0 \to E_{r-1} \to E_r \to P \to 0$, obtaining the long exact sequence
\begin{equation}\label{les:hom}
  0 \to \mathrm{Hom}_A(P, E_{r-1})
    \xrightarrow{(\iota_r)_*}
    \mathrm{Hom}_A(P, E_r)
    \xrightarrow{(\pi_r)_*}
    \mathrm{Hom}_A(P, P)
    \xrightarrow{\delta}
    \mathrm{Ext}^1_A(P, E_{r-1})
    \to \cdots
\end{equation}

\smallskip\noindent\textbf{Claim:} The map $(\pi_r)_* \colon
\mathrm{Hom}_A(P, E_r) \to \mathrm{Hom}_A(P, P)$ is the zero map.

\smallskip\noindent\textit{Proof of Claim.}
Since $P$ is a line bundle on a proper variety,
$\mathrm{Hom}_A(P, P) \cong k \cdot \mathrm{id}_P$.
If $(\pi_r)_*$ were nonzero, it would be surjective, so there would
exist $\varphi \in \mathrm{Hom}_A(P, E_r)$ with $\pi_r \circ \varphi = \mathrm{id}_P$.
This provides a right splitting of the extension, contradicting the
assumption that the extension class $\xi_r \neq 0$ is non-split.
Hence $(\pi_r)_* = 0$.
\hfill$\square$

\smallskip
By exactness of \eqref{les:hom} at $\mathrm{Hom}_A(P, E_r)$, the map
$(\iota_r)_*$ is surjective, so there is an isomorphism
\[
  \mathrm{Hom}_A(P, E_{r-1}) \xrightarrow{\;\sim\;} \mathrm{Hom}_A(P, E_r).
\]
Under the identification $\mathrm{Hom}_A(P, E_r) \cong H^0(A, U_r)$, this gives
$h^0(A, U_r) = h^0(A, U_{r-1}) = 1$ by the inductive hypothesis.
\end{proof}

\begin{lemma}[Indecomposability of $E_r$]\label{lem:indecomposable-Er}
For every integer $r \geq 1$, the vector bundle $E_r$ is indecomposable.
\end{lemma}

\begin{proof}
We pass to the Fourier--Mukai side and show indecomposability there.

\smallskip\noindent\textbf{Step 1: The finite-length module $\FF_r$.}
By Claim~\ref{claim:FM-Us}, the Fourier–Mukai transform satisfies
$\Phi_{\mathcal{P}}(U_r) \simeq \FF_r[-g]$ in $D^b(\widehat{A})$,
where $\FF_r$ is a coherent sheaf of finite length $r$ supported
set-theoretically at $\{0_{\widehat{A}}\}$.
Let $R := \mathcal{O}_{\widehat{A},\, 0}$ be the regular local ring of dimension $g$
at the origin of $\widehat{A}$, with maximal ideal $\mathfrak{m}$ and residue
field $\kappa = R/\mathfrak{m} \cong k$.
We regard $\FF_r$ as a finitely generated $R$-module of finite length $r$.

\smallskip\noindent\textbf{Step 2: The Betti number $\beta_g(\FF_r)$.}
By the derived fiber formula (Claim~\ref{claim:FM-Us}, Step 4),
\[
  H^i(A, U_r) \;\cong\; \mathrm{Tor}_{g-i}^R(\FF_r, \kappa)
  \quad\text{for all }i,
\]
and in particular
\[
  \beta_g(\FF_r) := \dim_k \mathrm{Tor}_g^R(\FF_r, \kappa)
                  = \dim_k H^0(A, U_r).
\]
By Proposition~\ref{prop:hom-Er-P}, $h^0(A, U_r) = 1$, so
\begin{equation}\label{eq:beta0}
  \beta_g(\FF_r) = 1.
\end{equation}

\smallskip\noindent\textbf{Step 3: Indecomposability of $\FF_r$.}
Suppose for contradiction that $\FF_r \cong M \oplus N$ for nonzero
$R$-modules $M$ and $N$.
Since $\FF_r$ has finite length, both $M$ and $N$ have finite length
and hence depth zero over $R$.  By the Auslander-Buchsbaum formula over
the regular local ring $R$ of depth $g$,
\[
  \mathrm{pd}_R(M) = g, \qquad \mathrm{pd}_R(N) = g.
\]
In particular $M \neq 0$ and $N \neq 0$ force
$\beta_g(M) \geq 1$ and $\beta_g(N) \geq 1$. By additivity of the $g$-th
Betti number,
\[
  \beta_g(\FF_r) = \beta_g(M) + \beta_g(N) \;\geq\; 2.
\]
This contradicts \eqref{eq:beta0}. Therefore $\FF_r$ is indecomposable
as an $R$-module.

\medskip
\noindent\textbf{Step 4: Descent to $E_r$.}
The Fourier--Mukai equivalence and
$\Phi_{\mathcal P}(U_r)\simeq \FF_r[-g]$ give
\[
\operatorname{End}_A(U_r)
\simeq
\operatorname{End}_{D^b(\widehat A)}(\FF_r[-g])
\simeq
\operatorname{End}_{\widehat A}(\FF_r).
\]
Since $\FF_r$ is indecomposable, its endomorphism ring
has no nontrivial idempotents. The same therefore
holds for $\operatorname{End}_A(U_r)$, so $U_r$
is indecomposable. Finally, tensoring by the line
bundle $P$ preserves indecomposability; hence
$E_r=P\otimes U_r$ is indecomposable.
\end{proof}
\begin{lemma}[ACM property of $E_r$] \label{lem:ACM_Er}
For every integer $r \ge 1$, the vector bundle $E_r$ is arithmetically Cohen-Macaulay with respect to $L$. Furthermore, for the untwisted case, $H^i(A, E_r) = 0$ for all $0 \le i \le g$.
\end{lemma}

\begin{proof}
We proceed by induction on $r$.

\noindent \textbf{Base case ($r=1$):} By definition, $E_1 = P \in \mathrm{Pic}^0(A) \setminus \{\mathcal{O}_A\}$. By Theorem~\ref{thm:main}, $P$ is ACM with respect to $L$, hence $H^i(A, P \otimes L^{\otimes m}) = 0$ for all $m \in \mathbb{Z}$ and $0 < i < g$. Furthermore, the Fourier--Mukai acyclicity established in
Theorem~\ref{thm:pic0-acyclicity} yields
\[
H^{i}(A,P)=0
\qquad
\text{for all }0\leq i\leq g.
\]

\noindent \textbf{Inductive step ($r \ge 2$):} Assume that the statement holds for $E_{r-1}$. Tensoring the defining short exact sequence 
\[
0 \longrightarrow E_{r-1} \longrightarrow E_r \longrightarrow P \longrightarrow 0
\]
with the locally free sheaf $L^{\otimes m}$ yields the exact sequence
\[
0 \longrightarrow E_{r-1} \otimes L^{\otimes m} \longrightarrow E_r \otimes L^{\otimes m} \longrightarrow P \otimes L^{\otimes m} \longrightarrow 0.
\]
The associated long exact sequence in cohomology contains the segment:
\[
\dots \longrightarrow H^i(A, E_{r-1} \otimes L^{\otimes m}) \longrightarrow H^i(A, E_r \otimes L^{\otimes m}) \longrightarrow H^i(A, P \otimes L^{\otimes m}) \longrightarrow \dots
\]
For any $m \in \mathbb{Z}$ and $0 < i < g$, the left term vanishes by the inductive hypothesis, and the right term vanishes because $P$ is ACM. Exactness therefore forces the middle term to vanish:
\[
H^i(A, E_r \otimes L^{\otimes m}) = 0.
\]
To verify total vanishing for the untwisted case ($m=0$), we examine the long exact sequence at degree $i$ for the full range $0 \le i \le g$:
\[
\dots \longrightarrow H^i(A, E_{r-1}) \longrightarrow H^i(A, E_r) \longrightarrow H^i(A, P) \longrightarrow \dots
\]
The left term vanishes by the inductive hypothesis, and the right term
vanishes by Theorem~\ref{thm:pic0-acyclicity}. This completes the induction.
\end{proof}

\begin{theorem}[Indecomposable Rank-$r$ ACM Bundle]\label{thm:main2}
For every integer $r\ge 1$, the bundle $E_r$ of Definition~\ref{def:iter} satisfies:
\begin{enumerate}
  \item $\operatorname{rk}(E_r)=r$ and $E_r$ is indecomposable;
  \item $c_1(E_r)=0$ in $\operatorname{NS}(A)$;
  \item $E_r$ is ACM with respect to $L$, that is,
        $H^i\bigl(A,E_r\otimes L^{m}\bigr)=0$ for all $m\in\mathbb Z$ and all $0<i<g$.
\end{enumerate}
\end{theorem}
\begin{proof}
The theorem follows from our preceding lemmas:
\begin{enumerate}
    \item Rank and indecomposability are established by Lemma~\ref{lem:chern} and Lemma~\ref{lem:indecomposable-Er}, respectively.
    \item The vanishing of the first Chern class is given by Lemma~\ref{lem:chern}.
    \item The ACM property is verified directly by Lemma \ref{lem:ACM_Er}, which demonstrates that $H^i(A, E_r \otimes L^{\otimes m}) = 0$ for all $m \in \mathbb{Z}$ and $0 < i < g$.
\end{enumerate}
\end{proof}
\begin{proposition}[Non-Ulrichness of the constructed bundles]
\label{prop:constructed-bundles-not-ulrich}
Let $(A,L)$ be a polarized abelian variety of dimension
$g\geq 2$, and let
$P\in\operatorname{Pic}^0(A)\setminus\{\mathcal O_A\}$.
For every $r\geq 1$ and every $a\in\mathbb Z$,
the bundle
\[
E_r\otimes L^{\otimes a}
\]
is not Ulrich with respect to $L$.
In particular, $E_r$ is an indecomposable ACM bundle
which is not Ulrich.
\end{proposition}

\begin{proof}
By construction, $E_r$ admits a filtration with
$r$ successive quotients isomorphic to $P$.
Tensoring this filtration with $L^{\otimes m}$
and using the additivity of the Euler characteristic
gives
\[
\chi\!\left(A,E_r\otimes L^{\otimes m}\right)
=
r\,\chi\!\left(A,P\otimes L^{\otimes m}\right)
\qquad (m\in\mathbb Z).
\]
Since $P$ is algebraically trivial, Riemann--Roch
on the abelian variety $A$ yields
\[
\chi\!\left(A,P\otimes L^{\otimes m}\right)
=
m^g\chi(A,L).
\]
Consequently,
\begin{equation}
\label{eq:constructed-bundles-euler}
\chi\!\left(A,E_r\otimes L^{\otimes m}\right)
=
r\,m^g\chi(A,L).
\end{equation}
Here
\[
\chi(A,L)=\frac{(L^g)}{g!}>0
\]
because $L$ is ample.

Suppose that $E_r\otimes L^{\otimes a}$ were Ulrich.
By the cohomological definition of an Ulrich bundle,
\[
H^i\!\left(
A,E_r\otimes L^{\otimes(a-j)}
\right)=0
\]
for all $0\leq i\leq g$ and all $1\leq j\leq g$.
In particular, its Euler characteristic must vanish
for each such $j$. By
\eqref{eq:constructed-bundles-euler}, this gives
\[
0=r(a-j)^g\chi(A,L),
\qquad 1\leq j\leq g.
\]
Since $r\geq1$ and $\chi(A,L)>0$, we obtain
$a=j$ for every $1\leq j\leq g$.
As $g\geq2$, the cases $j=1$ and $j=2$ contradict
each other. Thus no such twist is Ulrich.

Taking $a=0$ proves that $E_r$ itself is not Ulrich.
\end{proof}
\begin{proposition}[Semistability of $E_r$] \label{lem:semistability_Er}
For every integer $r \ge 1$, the vector bundle $E_r$ is semistable with respect to the polarization $L$, with slope $\mu_L(E_r) = 0$. Furthermore, for $r \ge 2$, $E_r$ is strictly semistable.
\end{proposition}

\begin{proof}
By Lemma \ref{lem:chern}, $c_1(E_r) = 0$ in $\mathrm{NS}(A)$ for all $r \ge 1$, hence the slope is trivially $\mu_L(E_r) = 0$. We prove semistability by induction on $r$.

\noindent \textbf{Base case ($r=1$):} By definition, $E_1 = P$. As a line bundle of degree zero, $P$ is slope-stable, and thus semistable.

\noindent \textbf{Inductive step ($r \ge 2$):} Assume $E_{r-1}$ is semistable of slope zero. By definition, $E_r$ fits into the short exact sequence:
\[
0 \longrightarrow E_{r-1} \longrightarrow E_r \longrightarrow P \longrightarrow 0.
\]
It is a standard fundamental property of slope stability that an extension of two semistable vector bundles of the same slope $\mu$ is again semistable of slope $\mu$ (see e.g. \cite{HL}). Since the outer terms $E_{r-1}$ and $P$ are both semistable with slope zero, it follows immediately that their extension $E_r$ is semistable with $\mu_L(E_r) = 0$.

To establish that $E_r$ is strictly semistable (i.e., not stable) for $r \ge 2$, observe that the defining exact sequence exhibits $E_{r-1}$ as a proper subbundle of $E_r$. Since $0 < \mathrm{rk}(E_{r-1}) < \mathrm{rk}(E_r)$ and $\mu_L(E_{r-1}) = \mu_L(E_r) = 0$, the subbundle $E_{r-1}$ prevents $E_r$ from satisfying the strict inequality required for slope stability. Therefore, $E_r$ is strictly semistable.
\end{proof}

\section{Explicit Families and Homological Invariants}\label{sec:construction}
In the preceding section, we proved the existence of indecomposable
ACM bundles of every rank. We now make a particular construction
explicit and study its isomorphism classes and homological invariants.
Throughout this section, $(A,L)$ is a polarized abelian variety of
dimension $g\ge 2$ over an algebraically closed field $k$ of
characteristic zero, and $\widehat A=\operatorname{Pic}^0(A)$.

\begin{proposition}
\label{prop:explicit-unipotent-construction}
Fix \(P\in\operatorname{Pic}^0(A)\setminus\{\mathcal O_A\}\) and
\(0\ne\eta\in H^1(A,\mathcal O_A)\). There exist unipotent
bundles $U_r$ of rank $r$, for all $r\geq 1$, with
$U_1=\mathcal O_A$ and $h^0(A,U_r)=1$, such that both sequences
\begin{align}
0&\longrightarrow\mathcal O_A\longrightarrow U_r
  \longrightarrow U_{r-1}\longrightarrow0,
  \label{eq:explicit-line-kernel}\\
0&\longrightarrow U_{r-1}\longrightarrow U_r
  \longrightarrow\mathcal O_A\longrightarrow0
  \label{eq:explicit-line-quotient}
\end{align}
are nonsplit for every $r\geq2$.
The bundles $E_r:=P\otimes U_r$ are indecomposable and ACM with
respect to $L$, and satisfy
\[
H^i(A,E_r)=0\quad(0\leq i\leq g),
\qquad \det(E_r)\simeq P^{\otimes r}.
\]
\end{proposition}

\begin{proof}
Choose an affine open cover $\{V_\alpha\}$ of $A$ and a
\v{C}ech cocycle $\{a_{\alpha\beta}\}$ representing $\eta$, with
$a_{\alpha\beta}+a_{\beta\gamma}=a_{\alpha\gamma}$.
Let $N_r$ be the nilpotent Jordan block with ones on the
superdiagonal, and define
\[
G_{\alpha\beta}^{(r)}
:=\exp(a_{\alpha\beta}N_r)
=\sum_{m=0}^{r-1}\frac{a_{\alpha\beta}^{m}N_r^m}{m!}.
\]
The sum is finite and the denominators are invertible since
$\operatorname{char}k=0$. Moreover,
\[
G_{\alpha\beta}^{(r)}G_{\beta\gamma}^{(r)}
=\exp((a_{\alpha\beta}+a_{\beta\gamma})N_r)
=G_{\alpha\gamma}^{(r)}.
\]
Thus these matrices define a vector bundle $U_r$.

The invariant subspaces spanned by the first $i$ standard basis
vectors give a filtration
\[
0=U_0\subset U_1\subset\cdots\subset U_r,
\qquad U_i/U_{i-1}\simeq\mathcal O_A.
\]
The restriction to the first $r-1$ basis vectors has transition
matrices $G_{\alpha\beta}^{(r-1)}$, giving
\eqref{eq:explicit-line-quotient}. Quotienting by the first basis
vector gives the same matrices of size $r-1$, yielding
\eqref{eq:explicit-line-kernel}.

Pull back \eqref{eq:explicit-line-kernel} along the distinguished
inclusion $\mathcal O_A=U_1\hookrightarrow U_{r-1}$.
Its inverse image is the subbundle spanned by the first two basis
vectors, whose transition matrices are
\[
\begin{pmatrix}1&a_{\alpha\beta}\\0&1\end{pmatrix}.
\]
Hence this pullback is the extension represented by $\eta$ and is
nonsplit. In particular, \eqref{eq:explicit-line-kernel} is nonsplit
for every $r\geq2$, not only for $r=2$.

We next prove $h^0(A,U_r)=1$ by induction. This holds for
$U_1=\mathcal O_A$. If it holds for $U_{r-1}$, its global sections
are generated by the distinguished inclusion
$\mathcal O_A\hookrightarrow U_{r-1}$. By the pullback calculation,
the connecting map for \eqref{eq:explicit-line-kernel},
\[
\delta_r:H^0(A,U_{r-1})\longrightarrow H^1(A,\mathcal O_A),
\]
sends this generator to the nonzero class $\eta$.
It is therefore injective, and the cohomology sequence gives
$H^0(A,U_r)\simeq H^0(A,\mathcal O_A)=k$.
If \eqref{eq:explicit-line-quotient} were split, it would instead
give $h^0(A,U_r)=h^0(A,U_{r-1})+1=2$, a contradiction.

To prove indecomposability, suppose $U_r\simeq V\oplus W$ with
$V,W\neq0$. For the projection $U_r\to V$, choose the least $i$
for which its restriction to $U_i$ is nonzero. It then induces
a nonzero map $U_i/U_{i-1}\simeq\mathcal O_A\to V$, so
$h^0(A,V)\geq1$. The same argument gives $h^0(A,W)\geq1$,
contradicting $h^0(A,U_r)=1$. Thus $U_r$ is indecomposable,
and so is $E_r=P\otimes U_r$.

Finally, $E_r$ has a filtration with successive quotients $P$.
The rank-one vanishing established earlier says that $P$ is
acyclic and ACM with respect to $L$. Successive applications
of the long exact cohomology sequence, also after tensoring
with $L^{\otimes m}$ for every $m\in\mathbb Z$, prove the
asserted vanishings and the ACM property of $E_r$.
The filtration also gives $\det(E_r)\simeq P^{\otimes r}$.
\end{proof}

Tensoring \eqref{eq:explicit-line-quotient} with $P$ gives
\[
0\longrightarrow E_{r-1}\longrightarrow E_r
\longrightarrow P\longrightarrow0,
\]
so this construction makes explicit a sequence of the type used
earlier. Tensoring \eqref{eq:explicit-line-kernel} with $P$ also gives
\[
0\longrightarrow P\longrightarrow E_r
\longrightarrow E_{r-1}\longrightarrow0.
\]
Both sequences are nonsplit, since tensoring with a line bundle
is an exact equivalence.

\begin{remark}[On the characteristic assumption]
The exponential construction above is valid in every rank
when $\operatorname{char}k=0$. If $\operatorname{char}k=p>0$,
the same construction and proof remain valid for $r\leq p$,
since all factorials $m!$ with $0\leq m\leq r-1$ are invertible
in $k$. For $r>p$, the displayed exponential formula is
not defined as written.

This restriction concerns the explicit construction, rather
than the underlying existence of ACM bundles.
The rank-one ACM statement remains valid in arbitrary
characteristic, using the vanishing theorem for ample line
bundles on abelian varieties in place of Kodaira vanishing.
The Fourier--Mukai correspondence with finite-length modules
also provides indecomposable unipotent bundles of every rank
in arbitrary characteristic.
\end{remark}

\begin{theorem}[Moduli and homological invariants of the bundles $E_r(\eta,P)$]\label{thm:moduli-explicit}
Let $(A,L)$ be a polarized abelian variety of dimension $g\ge2$ over an algebraically
closed field $k$ of characteristic zero. For $P\in\operatorname{Pic}^0(A)\setminus\{\mathcal O_A\}$,
$0\neq\eta\in H^1(A,\mathcal O_A)$ and $r\ge1$, let $U_r(\eta)$ and
$E_r(\eta,P):=P\otimes U_r(\eta)$ be the vector bundles of Proposition~\ref{prop:explicit-unipotent-construction}. Then:
\begin{enumerate}
\item $\operatorname{End}_A(E_r(\eta,P))\cong k[t]/(t^r)$; in particular it is a local
  Artinian $k$-algebra.
\item For every $0\le i\le g$,
  \[
   \dim_k\operatorname{Ext}^i_A(P,E_r(\eta,P))=\dim_k\operatorname{Ext}^i_A(E_r(\eta,P),P)=\binom gi ;
  \]
  in particular these dimensions are independent of $r$ and $\eta$.
\item Let $r,s\ge2$, $Q\in\operatorname{Pic}^0(A)\setminus\{\mathcal O_A\}$ and
  $0\neq\eta'\in H^1(A,\mathcal O_A)$. Then $E_r(\eta,P)\cong E_s(\eta',Q)$ if and only if
  $r=s$, $P\cong Q$ and $[\eta]=[\eta']$ in $\mathbb P(H^1(A,\mathcal O_A))\cong\mathbb P^{g-1}$.
  Consequently, for every $r\ge2$ the bundles $E_r(\eta,P)$ form a $(2g-1)$-dimensional
  family, parametrized by
  $\bigl(\operatorname{Pic}^0(A)\setminus\{\mathcal O_A\}\bigr)\times\mathbb P(H^1(A,\mathcal O_A))$,
  of pairwise non-isomorphic indecomposable ACM bundles.
\end{enumerate}
\end{theorem}

\begin{proof}
Let \(N_r\) denote the nilpotent Jordan block of size \(r\) with ones
on the superdiagonal, as in Proposition~\ref{prop:explicit-unipotent-construction}.

\medskip
\noindent\textit{The Fourier--Mukai module of \(U_r(\eta)\).}
Choose an affine open cover \(\{V_\lambda\}\) of \(A\) and a
\v{C}ech cocycle \(\{a_{\lambda\mu}\}\) representing \(\eta\).
The transition matrices of \(U_r(\eta)\) are
\[
   G_{\lambda\mu}=\exp(a_{\lambda\mu}N_r).
\]
If \(a'_{\lambda\mu}=a_{\lambda\mu}+b_\mu-b_\lambda\), then
\[
   \exp(a'_{\lambda\mu}N_r)
   =\exp(-b_\lambda N_r)\,G_{\lambda\mu}\,\exp(b_\mu N_r).
\]
Thus replacing the cocycle by a cohomologous one changes only the
local frames. Passing to a common refinement shows that the
isomorphism class of \(U_r(\eta)\) depends only on \(\eta\).

For \(n\geq 1\), put \(T_n=\operatorname{Spec} k[t]/(t^n)\).
Since \(\operatorname{char} k=0\), the truncated exponentials
\(\exp(t a_{\lambda\mu})\) satisfy the cocycle condition and define
a line bundle \(\mathcal L_n\) on \(A\times T_n\). These line
bundles are compatible under restriction and have special fibre
\(\mathcal O_A\). By the defining property of the dual abelian
variety, they determine compatible morphisms
\[
   \gamma_n:T_n\longrightarrow\widehat A,\qquad
   \gamma_n(0)=\widehat 0.
\]
Writing \(\mathcal P\) for the normalized Poincar\'e bundle, we have
\[
   \mathcal L_n\simeq
   (1_A\times\gamma_n)^*\mathcal P,
\]
since any line-bundle factor pulled back from the local scheme
\(T_n\) is trivial. The morphisms \(\gamma_n\) define a formal
curve \(\gamma_\eta\) through \(\widehat 0\). Under the canonical
identification
\(T_{\widehat 0}\widehat A\simeq H^1(A,\mathcal O_A)\),
its tangent vector is \(\eta\), since its transition functions
modulo \(t^2\) are \(1+t a_{\lambda\mu}\).

Let \(p:A\times T_r\to A\) be the projection. In the ordered basis
\(t^{r-1},t^{r-2},\ldots,1\) of \(k[t]/(t^r)\), multiplication
by \(t\) has matrix \(N_r\). Thus the transition matrices of
\(p_*\mathcal L_r\) are \(\exp(a_{\lambda\mu}N_r)\), and
\[
   p_*\mathcal L_r\simeq U_r(\eta).
\]
Let \(\Psi\) denote the reverse Fourier--Mukai transform with kernel
\(\mathcal P\), as in Theorem~A.1. Flat base change along the
projection \(A\times\widehat A\to\widehat A\), followed by the
projection formula for the finite morphism \(1_A\times\gamma_r\),
gives
\[
   \Psi\bigl((\gamma_r)_*\mathcal O_{T_r}\bigr)
   \simeq Rp_*\bigl((1_A\times\gamma_r)^*\mathcal P\bigr)
   \simeq p_*\mathcal L_r
   \simeq U_r(\eta),
\]
where \(p\) is finite. Fourier--Mukai inversion, with the
conventions of Theorem~A.1, therefore yields
\[
   F_r(\eta):=\Phi_g(U_r(\eta))
   \simeq (-1_{\widehat A})^*(\gamma_r)_*\mathcal O_{T_r}
   \simeq (-\gamma_r)_*\mathcal O_{T_r}.
\]

\medskip
\noindent\textit{The local normal form.}
Put \(R_0=\mathcal O_{\widehat A,\widehat 0}\) and
\[
   R=\widehat{R_0}\simeq k[[x_1,\ldots,x_g]],
\]
with maximal ideal \(\mathfrak m\). We identify a finite-length
sheaf supported at \(\widehat 0\) with the corresponding
finite-length \(R\)-module, as in Theorem~\ref{thm:appendix-equivalence}. Under this
identification, \(F_r(\eta)\) is \(k[t]/(t^r)\) with \(R\)-action
induced by
\[
   \varphi_{\eta,r}:R
   \xrightarrow{\ \varphi_\eta\ } k[[t]]
   \longrightarrow k[t]/(t^r),
\]
where \(\varphi_\eta\) is the continuous local homomorphism
associated with \(-\gamma_\eta\). Set
\(I_r(\eta)=\ker(\varphi_{\eta,r})\).

The induced map on cotangent spaces
\[
   \mathfrak m/\mathfrak m^2\longrightarrow (t)/(t^2)
\]
is nonzero, since the tangent vector of \(-\gamma_\eta\) is
\(-\eta\neq 0\). Consequently, there exists a regular parameter
\(z_1\in\mathfrak m\) such that \(\varphi_\eta(z_1)\) has order
one. Using \(\varphi_\eta(z_1)\) as the parameter on the formal
disc, we may write \(\varphi_\eta(z_1)=t\). This change of
parameter preserves \((t^r)\) and hence preserves \(I_r(\eta)\).
It follows that \(\varphi_{\eta,r}\) is surjective, since its
image contains the class of \(t\).

Extend \(z_1\) to a regular system of parameters
\(z_1,w_2,\ldots,w_g\) of \(R\). For \(2\leq j\leq g\), put
\[
   h_j(t)=\varphi_\eta(w_j),\qquad
   z_j=w_j-h_j(z_1).
\]
Each \(h_j\) has zero constant term. Modulo \(\mathfrak m^2\),
this change subtracts a scalar multiple of \(z_1\) from each
\(w_j\), so \(z_1,\ldots,z_g\) is again a regular system of
parameters. In these coordinates, \(\varphi_\eta\) sends \(z_1\)
to \(t\) and \(z_2,\ldots,z_g\) to zero. The quotient map
\(R/(z_2,\ldots,z_g)\simeq k[[z_1]]\to k[[t]]\) is therefore
an isomorphism. Reducing modulo \(t^r\) gives
\[
   I_r(\eta)=(z_1^r,z_2,\ldots,z_g),\qquad
   F_r(\eta)\simeq R/I_r(\eta).
\]

\medskip
\noindent\textit{(1) Endomorphisms.}
Tensoring by \(P\) preserves endomorphism algebras, and
\(\Phi_g\) is an equivalence on unipotent bundles by Theorem~\ref{thm:appendix-equivalence}.
Every \(R\)-linear endomorphism of \(R/I_r(\eta)\) is multiplication
by the image of \(1\) under that endomorphism. Hence
\[
   \operatorname{End}_A(E_r(\eta,P))
   \simeq \operatorname{End}_A(U_r(\eta))
   \simeq \operatorname{End}_R(R/I_r(\eta))
   \simeq R/I_r(\eta)
   \simeq k[t]/(t^r).
\]
In particular, this is a local Artinian \(k\)-algebra.

\medskip
\noindent\textit{(2) Extension groups.}
The elements \(z_1^r,z_2,\ldots,z_g\) form a system of parameters
of the regular local ring \(R\), hence a regular sequence. Since
they belong to \(\mathfrak m\), their Koszul complex is a minimal
free resolution of \(F_r(\eta)\). Thus
\[
   \dim_k\operatorname{Tor}^R_j(F_r(\eta),k)
   =\binom gj
   \qquad (0\leq j\leq g).
\]
By the derived fibre formula at \(\widehat 0\)
(Lemma~\ref{lem:existence}, Step~4) and the identity
\(\Phi(U_r(\eta))\simeq F_r(\eta)[-g]\), we obtain
\[
   H^i(A,U_r(\eta))
   \simeq \operatorname{Tor}^R_{g-i}(F_r(\eta),k).
\]
Here the Tor groups may be computed over \(R\), since completion
is flat and preserves the finite-length module and the residue
field. As \(P\) is invertible, it follows that
\[
   \dim_k\operatorname{Ext}^i_A(P,E_r(\eta,P))
   =h^i(A,U_r(\eta))
   =\binom g{g-i}=\binom gi.
\]
Finally, Serre duality and \(\omega_A\simeq\mathcal O_A\) give
\[
   \operatorname{Ext}^i_A(E_r(\eta,P),P)
   \simeq
   \operatorname{Ext}^{g-i}_A(P,E_r(\eta,P))^\vee,
\]
which has the same dimension.

\medskip
\noindent\textit{(3) Isomorphism classes.}
Suppose that \(E_r(\eta,P)\simeq E_s(\eta',Q)\), where \(r,s\geq2\).
Comparing ranks gives \(r=s\). Tensoring by \(Q^{-1}\) yields
\[
   (P\otimes Q^{-1})\otimes U_r(\eta)\simeq U_r(\eta').
\]
The right-hand side has a one-dimensional space of global sections
by Proposition~\ref{prop:explicit-unipotent-construction}. If \(P\otimes Q^{-1}\) were nontrivial, the
left-hand side would be acyclic by the same proposition. Hence
\(P\simeq Q\), and consequently \(U_r(\eta)\simeq U_r(\eta')\).

Applying \(\Phi_g\) gives an isomorphism of \(R\)-modules
\[
   R/I_r(\eta)\simeq R/I_r(\eta').
\]
Their annihilators are equal, so \(I_r(\eta)=I_r(\eta')\).
Since \(r\geq2\), we have \(z_1^r\in\mathfrak m^2\). The local
normal form therefore gives
\[
   \frac{I_r(\eta)+\mathfrak m^2}{\mathfrak m^2}
   =\langle\overline z_2,\ldots,\overline z_g\rangle
   =\ker\!\left(
      \mathfrak m/\mathfrak m^2\longrightarrow(t)/(t^2)
     \right),
\]
where the last map is induced by \(\varphi_\eta\). The annihilator
of this hyperplane in \(T_{\widehat 0}\widehat A\) is the tangent
line \(k\eta\). Thus equality of the ideals implies
\(k\eta=k\eta'\), or equivalently \([\eta]=[\eta']\).

Conversely, suppose that \(r=s\), \(P\simeq Q\), and
\(\eta'=c\eta\) for some \(c\in k^\times\). By independence of the
cocycle, we may represent \(\eta'\) by \(\{c a_{\lambda\mu}\}\).
The diagonal matrix
\[
   D(c)=\operatorname{diag}(c^{r-1},c^{r-2},\ldots,c,1)
\]
satisfies \(D(c)N_rD(c)^{-1}=cN_r\). It therefore conjugates the
transition matrices of \(U_r(\eta)\) to those of \(U_r(\eta')\).
This proves \(U_r(\eta)\simeq U_r(\eta')\), and tensoring by \(P\)
gives the required isomorphism of the corresponding \(E_r\).

\medskip
\noindent\textit{The family.}
Write \(V=H^1(A,\mathcal O_A)\), and let \(\mathbb P(V)\)
parametrize lines in \(V\). Choose a basis \(e_1,\ldots,e_g\) of
\(V\), represented on a common affine cover by cocycles
\(a^{(1)},\ldots,a^{(g)}\). Let \(x_1,\ldots,x_g\) be the
corresponding homogeneous coordinates. On the chart
\(B_\ell=\{x_\ell\neq0\}\), put
\[
   a_\ell=\sum_{j=1}^g\frac{x_j}{x_\ell}\,a^{(j)}.
\]
The matrices \(\exp((a_\ell)_{\lambda\mu}N_r)\) have regular
entries and define a vector bundle on \(A\times B_\ell\).
On \(B_\ell\cap B_m\), we have
\[
   a_\ell=c_{\ell m}a_m,\qquad c_{\ell m}=x_m/x_\ell.
\]
Conjugation by \(D(c_{\ell m})\) gives an isomorphism from the
bundle over \(B_m\) to the bundle over \(B_\ell\) on their
overlap. Since \(c_{\ell m}c_{mn}=c_{\ell n}\) and \(D\) is
multiplicative, these isomorphisms satisfy the cocycle condition.
The resulting vector bundle \(\mathcal U\) on \(A\times\mathbb P(V)\)
has fibre isomorphic to \(U_r(\eta)\) over \([\eta]\).

On \(A\times\mathbb P(V)\times\widehat A\), set
\[
   \mathcal E
   =\operatorname{pr}_{12}^{*}\mathcal U
      \otimes\operatorname{pr}_{13}^{*}\mathcal P.
\]
Restrict the last factor to \(\widehat A\setminus\{\widehat0\}\).
The fibre at \(([\eta],\alpha)\) is \(E_r(\eta,P_\alpha)\),
which is indecomposable and ACM by Proposition~\ref{prop:explicit-unipotent-construction}. The
isomorphism criterion above shows that distinct parameters give
non-isomorphic fibres. Finally, the parameter space has dimension
\((g-1)+g=2g-1\).
\end{proof}

\section{Representation Type of ACM Bundles}
Throughout this section, let \(A\) be an abelian variety of dimension \(g\)
over an algebraically closed field \(k\) of arbitrary characteristic,
and let \(\widehat A=\operatorname{Pic}^0(A)\) be its dual abelian
variety.

\subsection{Fourier--Mukai Equivalence for Unipotent Bundles}

 We establish the categorical link between unipotent bundles on an abelian variety $A$ and the module category of a specific regular local ring. This equivalence serves as the technical bridge between the geometry of the variety and the representation-theoretic definition of wildness.

Let $A$ be a complex abelian variety of dimension $g$, and let $\widehat{A} = \operatorname{Pic}^{0}(A)$ be its dual abelian variety. We denote by $\cP$ the normalized Poincar\'e line bundle on $A \times \widehat{A}$. The Fourier–Mukai functor $\Phi_\mathcal{P} \colon D^b(A) \to D^b(\widehat{A})$ is defined as:
\[
\Phi_\mathcal{P}(E) = \mathbf{R}p_{\widehat{A},*} \left( p_A^* E \otimes \cP \right).
\]
According to the seminal work of Mukai \cite{MukaiDuality}, $\Phi_\mathcal{P}$ is an equivalence of triangulated categories. For our purposes, we focus on the subcategory of unipotent bundles.

\begin{definition}\label{def:unipotent}
A vector bundle $U$ on $A$ is said to be \emph{unipotent} if it admits a filtration $0 = U_0 \subset U_1 \subset \dots \subset U_r = U$ such that each successive quotient $U_i / U_{i-1}$ is isomorphic to the structure sheaf $\cO_A$. We denote the category of such bundles by $\cU$.
\end{definition}

\begin{lemma}\label{lem:acm-unipotent}
Let $U$ be a unipotent vector bundle on $A$ and  $P \in \mathrm{Pic}^0(A)$ be a nontrivial
line bundle. Then the vector bundle $E := P \otimes U$ is arithmetically
Cohen--Macaulay $($ACM$)$ with respect to the polarization defined by $L$,
i.e.,
\[
  H^i(A,\, E \otimes L^n) = 0 \quad \text{for all } n \in \mathbb{Z}
  \text{ and } 0 < i < g.
\]
\end{lemma}

\begin{proof}
By Theorem~\ref{thm:main} and Remark~\ref{chf}, \(P\) is ACM with respect to \(L\).
Let \(n\in\mathbb Z\). Tensoring the unipotent filtration of \(U\)
by \(P\otimes L^{\otimes n}\) gives a filtration whose successive
quotients are all isomorphic to \(P\otimes L^{\otimes n}\). Each
quotient has vanishing intermediate cohomology. Applications of the
cohomology long exact sequence therefore give
\[
    H^i(A,P\otimes U\otimes L^{\otimes n})=0
    \qquad (0<i<g).
\]
This holds for every \(n\), so \(P\otimes U\) is ACM.
\end{proof}

\begin{theorem}\label{thm:unipotent-equivalence}
Let $\mathcal{U}$ be the category of unipotent vector bundles on $A$. There exists 
an equivalence of categories:
\[
\mathcal{U} \;\simeq\; \mathrm{flmod}\!\left(k[[x_1, \ldots, x_g]]\right),
\]
where $\mathrm{flmod}(R)$ denotes the category of finite-length modules over 
the ring $R$.
\end{theorem}

\begin{proof}
See Theorem~\ref{thm:appendix-equivalence} in the Appendix, where the full 
chain of equivalences
\[
\mathcal{U} \;\simeq\; \mathrm{Coh}_{\{0\}}(\widehat{A}) 
\;\simeq\; \mathrm{flmod}(\mathcal{O}_{\widehat{A},0})
\;\simeq\; \mathrm{flmod}(k[[x_1,\ldots,x_g]])
\]
is established via the Fourier--Mukai transform, the stalk equivalence at the 
origin, and the Cohen Structure Theorem.
\end{proof}

\begin{definition}[Wild representation type]\label{def:wild}
Let $\mathcal{C}$ be a $k$-linear exact category, such as a full subcategory of $\operatorname{Coh}(A)$ closed under extensions. Let $\operatorname{mod} k\langle u, v \rangle$ denote the category of finite-dimensional modules over the free associative $k$-algebra in two non-commuting variables. A \emph{representation embedding}
\[
    F : \operatorname{mod} k\langle u, v \rangle \longrightarrow \mathcal{C}
\]
is an exact $k$-linear functor that preserves indecomposability and reflects isomorphism classes, i.e., $F(M) \cong F(N) \implies M \cong N$. We say that $\mathcal{C}$ is of \emph{wild representation type} if such a representation embedding $F$ exists. For module categories of finite-dimensional algebras, this coincides with the standard notion of wildness due to Drozd~\cite{Drozd79}; see also~\cite[Chapter~XIX]{SSA07}.
\end{definition}

\begin{theorem}\label{thm:acm-wild}
Let $A$ be an abelian variety of dimension $g \geq 2$, and let $L$ be 
an ample line bundle on $A$. The category of arithmetically Cohen--Macaulay 
$($ACM$)$ vector bundles on $A$ with respect to the polarization defined by 
$L$ is of wild representation type.
\end{theorem}

\begin{proof}
Put \(K=k\langle u,v\rangle\) and
\[
    \Lambda=k[x_1,x_2]/(x_1^2,x_1x_2^2,x_2^3).
\]
By Drozd's theorem, \(\Lambda\) is wild \cite{Dr72}. Thus there
exists a \((\Lambda,K)\)-bimodule \(B\), free of finite rank as a right
\(K\)-module, for which \(B\otimes_K-\) is a representation embedding
from finite-dimensional \(K\)-modules to finite-dimensional
\(\Lambda\)-modules. Here we use wildness of a finite-dimensional
algebra in its usual bimodule form \cite[Chapter~XIX]{SSA07}.

Choose formal parameters \(x_1,\ldots,x_g\) at \(0_{\widehat A}\).
They give a surjection
\[
    \widehat{\mathcal O}_{\widehat A,0_{\widehat A}}
       \simeq k[[x_1,\ldots,x_g]] \twoheadrightarrow \Lambda
\]
by sending \(x_3,\ldots,x_g\) to zero and \(x_1,x_2\) to their
classes in \(\Lambda\). The cube of the maximal ideal maps to zero.
Because completion preserves the quotient by that cube, this
surjection defines a closed immersion
\(\jmath\colon Z=\operatorname{Spec}\Lambda\hookrightarrow\widehat A\)
supported at \(0_{\widehat A}\).

Fix \(P\in\operatorname{Pic}^0(A)\setminus\{\mathcal O_A\}\).
For a finite-dimensional \(K\)-module \(M\), let
\(\widetilde{B\otimes_K M}\) denote the sheaf on \(Z\) associated to
the \(\Lambda\)-module \(B\otimes_K M\). Define
\[
    F(M)=P\otimes
          \Psi_0\bigl(\jmath_*\widetilde{B\otimes_K M}\bigr),
\]
where \(\Psi_0\) is the inverse equivalence in Theorem~\ref{thm:appendix-equivalence}.
The pushforward \(\jmath_*\) is exact and fully faithful.
Theorem~\ref{thm:appendix-equivalence} makes \(\Psi_0\) an exact equivalence onto the category
of unipotent bundles. Tensoring by \(P\) is exact and fully faithful,
and Lemma~\ref{lem:acm-unipotent} shows that its image on unipotent bundles consists of
ACM bundles. Therefore \(F\) is exact and \(k\)-linear, preserves
indecomposability, and reflects isomorphism classes.
\end{proof}

\begin{remark}
In \cite{FaenziPonsLlopis} a variety is called CM-wild if it carries families of arbitrarily large dimension of pairwise non-isomorphic indecomposable ACM bundles; there, $X$ is
assumed to be ACM, which $A$ is not for $g\ge2$. Corollary~\ref{cor:large_ACM_families}
below shows that $(A,L)$ is wild in this sense as well.
\end{remark}

\begin{corollary}
\label{cor:large_ACM_families}
Let $(A,L)$ be a polarized abelian variety of dimension $g\geq 2$ over an algebraically closed field $k$. Then, for every integer $n\geq 1$, there exists an $n$-dimensional algebraic family of pairwise non-isomorphic indecomposable ACM bundles on $(A,L)$.
\end{corollary}

\begin{proof}
Fix an integer $n\geq 1$. We first construct an
$n$-dimensional family of pairwise non-isomorphic
indecomposable modules over the free associative algebra
$k\langle u,v\rangle$.

Let $J_n$ be the nilpotent Jordan block of size $n$,
given by
\[
J_n=
\begin{pmatrix}
0&1&0&\cdots&0\\
0&0&1&\cdots&0\\
\vdots&&\ddots&\ddots&\vdots\\
0&\cdots&0&0&1\\
0&\cdots&0&0&0
\end{pmatrix}.
\]

For every
\[
\mathbf a=(a_0,a_1,\ldots,a_{n-1})\in k^n,
\]
consider the $k\langle u,v\rangle$-module
$M_{\mathbf a}$ with underlying vector space $k^n$,
where the generators act by
\[
u\longmapsto J_n,
\qquad
v\longmapsto
a_0I_n+a_1J_n+\cdots+a_{n-1}J_n^{n-1}.
\]

We first show that $M_{\mathbf a}$ is indecomposable.
Indeed, the restriction of $M_{\mathbf a}$ to
the subalgebra $k[u]$ is isomorphic to
\[
k[u]/(u^n).
\]
This is an indecomposable $k[u]$-module.
Consequently, $M_{\mathbf a}$ is also
indecomposable as a
$k\langle u,v\rangle$-module.

Next, we show that different parameters give
non-isomorphic modules. Suppose that
\[
M_{\mathbf a}\cong M_{\mathbf b}
\]
for some $\mathbf a,\mathbf b\in k^n$.
Then there exists an invertible matrix $T$ satisfying
\[
TJ_n=J_nT
\]
and
\[
T\left(\sum_{i=0}^{n-1}a_iJ_n^i\right)
=
\left(\sum_{i=0}^{n-1}b_iJ_n^i\right)T.
\]
Since $T$ commutes with $J_n$, it commutes with
every polynomial in $J_n$. Hence
\[
\sum_{i=0}^{n-1}a_iJ_n^i
=
\sum_{i=0}^{n-1}b_iJ_n^i.
\]
The matrices
\[
I_n,J_n,\ldots,J_n^{n-1}
\]
are linearly independent. Therefore
\[
a_i=b_i
\qquad
\text{for every }0\leq i\leq n-1.
\]
Thus
\[
M_{\mathbf a}\cong M_{\mathbf b}
\quad\Longleftrightarrow\quad
\mathbf a=\mathbf b.
\]

Moreover, the matrices defining $M_{\mathbf a}$
depend polynomially on the parameters
$a_0,\ldots,a_{n-1}$. Hence these modules form
an algebraic family over
\[
S=\mathbb A_k^n.
\]
The underlying sheaf is
$\mathcal O_S^{\oplus n}$, so the family is
flat over $S$.

Use the notation of the proof of Theorem~\ref{thm:acm-wild}. Let \(\mathcal M\)
be the family of \(K\)-modules on \(S=\mathbb A_k^n\) constructed
above, and let
\[
    \mathcal N=B\otimes_K\mathcal M.
\]
The \(\Lambda\)-action makes \(\mathcal N\) a sheaf on \(Z\times S\).
Since \(B\) is free of finite rank as a right \(K\)-module, the
pushforward of \(\mathcal N\) by \(Z\times S\to S\) is a finite
direct sum of copies of the underlying locally free sheaf of
\(\mathcal M\). For every \(a\in S(k)\), the restriction of
\(\mathcal N\) to \(Z\times\{a\}\) corresponds to the module
\(B\otimes_K M_a\).

Let \(p\colon A\times Z\times S\to A\times S\) and
\(q\colon A\times Z\times S\to Z\times S\) be the projections.
Write \(\mathcal P_Z\) for the pullback of the Poincar\'e line
bundle to \(A\times Z\times S\), and let
\(\iota_S=(-1_A)\times\operatorname{id}_S\).
Define a sheaf on \(A\times S\) by
\[
    \mathcal E
      = \operatorname{pr}_A^*P\otimes
        \iota_S^*p_*\bigl(\mathcal P_Z\otimes q^*\mathcal N\bigr).
\]
The reduced subscheme of \(Z\) is \(\operatorname{Spec}k\). Thus
the projection \(A\times Z\to A\) identifies their underlying
topological spaces. We can cover \(A\) by open subsets \(V\)
such that the Poincar\'e bundle is trivial on \(V\times Z\).
Over each \(V\times S\), the pushforward in the displayed formula
is the pullback from \(S\) of the locally free pushforward of
\(\mathcal N\) by \(Z\times S\to S\). It follows that
\(\mathcal E\) is locally free on \(A\times S\) and flat over \(S\).

The projection \(p\) is finite, so its pushforward commutes with
restriction to each fibre over \(S\). By the definition of \(\Psi_0\)
in Theorem~A.1, this gives
\[
    \mathcal E|_{A\times\{a\}}
        \simeq P\otimes\Psi_0\bigl(
             \jmath_*\widetilde{B\otimes_K M_a}\bigr)
        = F(M_a).
\]
The properties of \(F\) proved in Theorem~6.5 show that these
fibres are indecomposable ACM bundles and that distinct
parameters have non-isomorphic fibres. Since \(\dim S=n\),
\(\mathcal E\) is the required algebraic family.
\end{proof}

\section{Stable ACM Bundles from Semihomogeneous Bundles}
We now apply Mukai's theory of simple semihomogeneous bundles
to obtain slope-stable ACM bundles with prescribed nonintegral
rational slope class and ranks determined by the polarization type.

\begin{theorem}[Stable ACM bundles from semi-homogeneous bundles]
\label{thm:stable-acm-semihomogeneous}
Let $(A,L)$ be a polarized abelian variety of dimension
$g\geq 2$ over an algebraically closed field of
characteristic zero. Let
\[
q=\frac{a}{b}\in\mathbb{Q}\setminus\mathbb{Z},
\qquad
b\geq 2,\qquad \gcd(a,b)=1.
\]
Then there exists a simple semi-homogeneous vector
bundle $E$ on $A$ such that
\[
\delta(E):=\frac{c_1(E)}{\operatorname{rk}(E)}
=q\,c_1(L)
\quad\text{in }\operatorname{NS}(A)_{\mathbb Q}.
\]
Every simple semi-homogeneous vector bundle with
this rational slope class is $\mu_L$-stable and
arithmetically Cohen--Macaulay with respect to $L$.

More precisely, if $L$ has polarization type
$(d_1,\ldots,d_g)$, then
\[
\operatorname{rk}(E)
=
\prod_{j=1}^{g}\frac{b}{\gcd(b,d_j)}.
\]
In particular,
\[
1\leq\operatorname{rk}(E)\leq b^g,
\qquad
\operatorname{rk}(E)\mid b^g.
\]
If $L$ is principally polarized, then
$\operatorname{rk}(E)=b^g$.
\end{theorem}

\begin{proof}
By Mukai's existence theorem
\cite[Theorem 7.11(1)]{MukaiSemiHomogeneous}, there exists a simple
semi-homogeneous vector bundle $E$ on $A$ satisfying
\[
\delta(E)=q\,c_1(L).
\]
Moreover, every simple semi-homogeneous vector bundle
is slope-stable with respect to any ample polarization
by \cite[Proposition 6.16]{MukaiSemiHomogeneous}. Thus $E$ is
$\mu_L$-stable.

We now prove that $E$ is ACM with respect to $L$.
Write
\[
q=m+q_0,
\qquad m\in\mathbb Z,\qquad 0<q_0<1,
\]
and set
\[
F:=E\otimes L^{-m}.
\]
Since tensoring by a power of $L$ preserves the ACM
property, it is enough to prove that $F$ is ACM.

The bundle $F$ is again simple and semi-homogeneous,
and
\[
\delta(F)=q_0c_1(L).
\]

We first recall a useful vanishing property.
Let $G$ be a simple semi-homogeneous vector bundle
whose rational slope class $\delta(G)$ is ample.
We claim that $G$ satisfies the $\mathrm{IT}_0$ condition.

Indeed, by Mukai's isogeny description
\cite[Theorem 5.8]{MukaiSemiHomogeneous}, there exist an abelian
variety $B$, an isogeny
\[
\pi:B\longrightarrow A,
\]
and a line bundle $M$ on $B$ such that
\[
G\cong\pi_*M.
\]
Since \(k\) has characteristic zero, \(\pi\) is finite \'etale.
The fibre product \(B\times_A B\) is the disjoint union of the
graphs of the translations \(t_x\) for \(x\in(\ker\pi)(k)\).
Base change along \(\pi\) gives
\[
    \pi^*G=\pi^*\pi_*M
        \simeq \bigoplus_{x\in(\ker\pi)(k)}t_x^*M.
\]
Taking first Chern classes gives
\[
\pi^*c_1(G)
=
\operatorname{rk}(G)c_1(M)
\quad\text{in }\operatorname{NS}(B)_{\mathbb Q}.
\]
Consequently,
\[
c_1(M)=\pi^*\delta(G).
\]
Since $\delta(G)$ is ample and $\pi$ is finite,
the line bundle $M$ is ample.

For every $\alpha\in\operatorname{Pic}^0(A)$,
the projection formula gives
\[
G\otimes P_\alpha
\cong
\pi_*(M\otimes\pi^*P_\alpha).
\]
Since $\pi$ is finite, its higher direct images
vanish. Hence
\[
H^i(A,G\otimes P_\alpha)
\cong
H^i(B,M\otimes\pi^*P_\alpha).
\]
The line bundle $M\otimes\pi^*P_\alpha$ is ample.
Since $\omega_B\cong\mathcal O_B$, Kodaira
vanishing implies
\[
H^i(B,M\otimes\pi^*P_\alpha)=0
\qquad\text{for all }i>0.
\]
Therefore
\[
H^i(A,G\otimes P_\alpha)=0
\qquad
\text{for all }i>0
\text{ and all }\alpha\in\operatorname{Pic}^0(A).
\]
Thus $G$ satisfies $\mathrm{IT}_0$.

We apply this observation to the twists of $F$.
We must show that
\[
H^i(A,F\otimes L^t)=0
\]
for every $t\in\mathbb Z$ and $1\leq i\leq g-1$.

First, suppose that $t\geq 0$. The bundle
$F\otimes L^t$ is simple and semi-homogeneous, with
\[
\delta(F\otimes L^t)
=
(q_0+t)c_1(L).
\]
Since $q_0+t>0$, this is an ample rational
divisor class. Therefore $F\otimes L^t$
satisfies $\mathrm{IT}_0$, and hence
\[
H^i(A,F\otimes L^t)=0
\qquad\text{for all }i>0.
\]

Next, suppose that $t\leq -1$. Consider
\[
G_t:=(F\otimes L^t)^\vee.
\]
This is again a simple semi-homogeneous bundle,
and its rational slope class is
\[
\delta(G_t)=(-q_0-t)c_1(L).
\]
Since $0<q_0<1$ and $t\leq -1$, we have
\[
-q_0-t>0.
\]
Thus $\delta(G_t)$ is ample, so $G_t$
satisfies $\mathrm{IT}_0$. In particular,
\[
H^j(A,G_t)=0
\qquad\text{for every }j>0.
\]

Using Serre duality and
$\omega_A\cong\mathcal O_A$, we obtain
\[
H^i(A,F\otimes L^t)
\cong
H^{g-i}(A,G_t)^\vee.
\]
For $1\leq i\leq g-1$, we have $g-i>0$.
It follows that
\[
H^i(A,F\otimes L^t)=0
\qquad(1\leq i\leq g-1).
\]

Combining the two cases, we obtain
\[
H^i(A,F\otimes L^t)=0
\quad
\text{for every }t\in\mathbb Z
\text{ and }1\leq i\leq g-1.
\]
Hence $F$ is ACM with respect to $L$.
Since $E=F\otimes L^m$, the bundle $E$ is
also ACM with respect to $L$.

It remains to determine the rank.
Let $(d_1,\ldots,d_g)$ be the polarization
type of $L$, and put
\[
K(aL):=\ker(\varphi_{L^{\otimes a}}),
\]
where for a negative integer $a$ the notation
$L^{\otimes a}$ denotes the corresponding
negative tensor power.

By Mukai's rank formula
\cite[Theorem 7.11(5)]{MukaiSemiHomogeneous},
\[
\operatorname{rk}(E)
=
\frac{b^g}
{\sqrt{|A[b]\cap K(aL)|}},
\]
where $A[b]$ is the group of $b$-torsion
points of $A$.

The polarization type gives
\[
|A[b]\cap K(aL)|
=
\prod_{j=1}^{g}\gcd(b,ad_j)^2.
\]
Since $\gcd(a,b)=1$, we have
\[
\gcd(b,ad_j)=\gcd(b,d_j).
\]
Therefore
\[
\operatorname{rk}(E)
=
\prod_{j=1}^{g}\frac{b}{\gcd(b,d_j)}.
\]
In particular, this rank divides $b^g$
and lies between $1$ and $b^g$.

If $L$ is principally polarized, then
$d_1=\cdots=d_g=1$, so
\[
\operatorname{rk}(E)=b^g.
\]
This completes the proof.
\end{proof}
\section{Further Cohomological Properties}

We next show that the Fourier--Mukai equivalence of Theorem~\ref{thm:unipotent-equivalence} yields a complete cohomological and homological description of all bundles in the wild subcategory $P \otimes \operatorname{Unip}(A)$, as well as a sharp criterion governing when tensor products, symmetric powers, and exterior powers preserve the ACM property.

\begin{proposition}[Universal Cohomology Table and Castelnuovo--Mumford Regularity]\label{prop:unip-cohom-reg}
Let $(A, L)$ be a polarized abelian variety of dimension $g \ge 2$, let $P \in \operatorname{Pic}^0(A) \setminus \{\mathcal{O}_A\}$, and let $U \in \operatorname{Unip}(A)$ be any nonzero unipotent vector bundle of rank $r \ge 1$. Set $E := P \otimes U$ and $h := \chi(A, L) = (L^g)/g!$. Then:
\begin{enumerate}
    \item[(1)] The cohomology dimensions of $E \otimes L^{\otimes m}$ depend only on the rank $r = \operatorname{rk}(U)$ and not on the isomorphism class of $U$:
    \[
        h^i(A, E \otimes L^{\otimes m}) = \begin{cases}
            r m^g h, & \text{if } m > 0 \text{ and } i = 0, \\
            r (-m)^g h, & \text{if } m < 0 \text{ and } i = g, \\
            0, & \text{otherwise.}
        \end{cases}
    \]
    \item[(2)] With respect to $L$, the bundle $E = P \otimes U$ has exact Castelnuovo--Mumford regularity
    \[
        \operatorname{reg}_L(E) = g,
    \]
    independent of the rank $r$ and of the unipotent filtration structure of $U$.
\end{enumerate}
\end{proposition}

\begin{proof}
(1) Since $U$ is an iterated extension of $r$ copies of $\mathcal{O}_A$, the bundle $E \otimes L^{\otimes m}$ is an iterated extension of $r$ copies of $P \otimes L^{\otimes m}$. By Theorem~\ref{thm:cohomology-table}, for every fixed $m \in \mathbb{Z}$, at most one cohomology group $H^i(A, P \otimes L^{\otimes m})$ is nonzero (namely $i = 0$ for $m > 0$, $i = g$ for $m < 0$, and none for $m = 0$). By induction on $r$ via the long exact cohomology sequence, $H^i(A, E \otimes L^{\otimes m}) = 0$ in all other degrees, and in the single nonvanishing degree the dimension is $r$ times that of $P \otimes L^{\otimes m}$.

(2) Recall that $E$ is $k$-regular with respect to $L$ in the sense of Castelnuovo--Mumford if $H^i(A, E \otimes L^{\otimes (k-i)}) = 0$ for all $i > 0$. By part (1), for $0 < i < g$ this vanishing holds for all $k \in \mathbb{Z}$, while for $i = g$ we have $h^g(A, E \otimes L^{\otimes (k-g)}) = 0$ if and only if $k - g \ge 0$, i.e., $k \ge g$. Thus $\operatorname{reg}_L(E) = g$.
\end{proof}

\begin{theorem}[Total Cohomological Nonvanishing of Unipotent Bundles and Torsion Dichotomy]\label{thm:torsion-dichotomy}
Let $(A, L)$ be a polarized abelian variety of dimension $g \ge 2$.
\begin{enumerate}
    \item[(1)] \textbf{(Total nonvanishing for unipotent bundles)} For every nonzero unipotent vector bundle $U \in \operatorname{Unip}(A)$,
    \[
        H^i(A, U) \ne 0 \quad \text{for every } 0 \le i \le g.
    \]
    In particular, no nonzero unipotent vector bundle is ACM with respect to any polarization on $A$.
    \item[(2)] \textbf{(Tensor product dichotomy)} Let $E_1 = P_1 \otimes U_1$ and $E_2 = P_2 \otimes U_2$, where $P_1, P_2 \in \operatorname{Pic}^0(A) \setminus \{\mathcal{O}_A\}$ and $U_1, U_2 \in \operatorname{Unip}(A)$ are nonzero. Then $E_1 \otimes E_2$ is ACM with respect to $L$ if and only if $P_1 \otimes P_2 \not\simeq \mathcal{O}_A$. In particular, for any $E = P \otimes U$, the endomorphism bundle $\mathcal{E}nd(E) \cong E^\vee \otimes E$ is never ACM, whereas the dual bundle $E^\vee \cong P^{-1} \otimes U^\vee$ is always ACM.
    \item[(3)] \textbf{(Torsion criterion for symmetric and exterior powers)} Let $E = P \otimes U$ with $P \in \operatorname{Pic}^0(A) \setminus \{\mathcal{O}_A\}$ and $U \in \operatorname{Unip}(A)$ of rank $r \ge 1$. For any integer $k \ge 1$ (and $1 \le k \le r$ in the exterior case), the bundles $\operatorname{Sym}^k(E)$ and $\bigwedge^k(E)$ are ACM with respect to $L$ if and only if $P^{\otimes k} \not\simeq \mathcal{O}_A$. Consequently, $\operatorname{Sym}^k(E)$ is ACM for all $k \ge 1$ if and only if $P$ is a non-torsion point of $\operatorname{Pic}^0(A)$.
\end{enumerate}
\end{theorem}

\begin{proof}
(1) Let $U \in \operatorname{Unip}(A)$ be nonzero, and let $\mathcal{F} := \Phi^g(U) \in \operatorname{flmod}(R)$ be its Fourier--Mukai transform over the $g$-dimensional regular local ring $R = \mathcal{O}_{\hat{A}, 0}$ with residue field $\kappa$. By the derived fiber formula, $H^i(A, U) \cong \operatorname{Tor}_{g-i}^R(\mathcal{F}, \kappa)$, so $h^i(A, U) = \beta_{g-i}(\mathcal{F})$ for every $0 \le i \le g$. Since $\mathcal{F}$ is a nonzero finite-length $R$-module, $\operatorname{depth}_R(\mathcal{F}) = 0$, and the Auslander--Buchsbaum formula gives $\operatorname{pd}_R(\mathcal{F}) = g$. In a minimal free resolution
\[
    0 \longrightarrow R^{\beta_g} \xrightarrow{d_g} R^{\beta_{g-1}} \xrightarrow{d_{g-1}} \cdots \xrightarrow{d_1} R^{\beta_0} \longrightarrow \mathcal{F} \longrightarrow 0,
\]
Write \(F_j=R^{\beta_j}\) for the free term in degree \(j\), with
\(F_j=0\) for \(j>g\). If \(F_j=0\) for some \(0\leq j<g\),
exactness makes \(F_{j+2}\to F_{j+1}\) surjective. By minimality,
the image lies in \(\mathfrak m F_{j+1}\), where \(\mathfrak m\)
is the maximal ideal of \(R\). Nakayama's lemma gives
\(F_{j+1}=0\). Repeating this argument gives \(F_g=0\), contrary
to \(\operatorname{pd}_R(F)=g\). Thus \(\beta_j>0\) for
\(0\leq j\leq g\), and \(H^i(A,U)\ne 0\) for \(0\leq i\leq g\). Since $g \ge 2$, we have $H^1(A, U) \ne 0$, so $U$ is not ACM.

(2) The category $\operatorname{Unip}(A)$ is closed under tensor products and duals because a tensor product or dual of filtrations with quotients $\mathcal{O}_A$ again has quotients $\mathcal{O}_A$. Thus $U := U_1 \otimes U_2 \in \operatorname{Unip}(A)$ is a nonzero unipotent bundle of rank $\operatorname{rk}(U_1)\operatorname{rk}(U_2) \ge 1$, and $E_1 \otimes E_2 \cong (P_1 \otimes P_2) \otimes U$. If $P_1 \otimes P_2 \not\simeq \mathcal{O}_A$, then $E_1 \otimes E_2$ is ACM by Lemma~\ref{lem:acm-unipotent}. Conversely, if $P_1 \otimes P_2 \simeq \mathcal{O}_A$, then $E_1 \otimes E_2 \cong U \in \operatorname{Unip}(A)$, which is not ACM by part (1). For $E = P \otimes U$, we have $\mathcal{E}nd(E) \cong U^\vee \otimes U \in \operatorname{Unip}(A)$ (non-ACM), while $E^\vee \cong P^{-1} \otimes U^\vee$ has $P^{-1} \not\simeq \mathcal{O}_A$ and $U^\vee \in \operatorname{Unip}(A)$, hence is ACM.

(3) Since $U \in \operatorname{Unip}(A)$ has rank $r$, both $\operatorname{Sym}^k(U)$ (for $k \ge 1$) and $\bigwedge^k(U)$ (for $1 \le k \le r$) are nonzero unipotent vector bundles, and we have canonical isomorphisms
\[
    \operatorname{Sym}^k(P \otimes U) \cong P^{\otimes k} \otimes \operatorname{Sym}^k(U), \qquad \bigwedge^k(P \otimes U) \cong P^{\otimes k} \otimes \bigwedge^k(U).
\]
Applying Lemma~\ref{lem:acm-unipotent} when $P^{\otimes k} \not\simeq \mathcal{O}_A$ and part (1) when $P^{\otimes k} \simeq \mathcal{O}_A$ yields the result.
\end{proof}

\section{Generic Vanishing, M-Regularity, and ACM Bundles}
Pareschi and Popa \cite{PP1} introduced \emph{Mukai regularity}
(\emph{M-regularity}) as an analogue, for sheaves on abelian varieties, of
Castelnuovo--Mumford regularity. It is defined through the Fourier--Mukai
transform and measures the positivity of a sheaf. In this section we relate
M-regularity, and the weaker generic vanishing condition, to the ACM
property. We first recall the definitions.

\begin{definition}\label{def:GV}
Let $F$ be a coherent sheaf on an abelian variety $A$ of dimension $g$.
For $i\ge 0$, the \emph{$i$-th cohomological support locus} of $F$ is
\[
  V^i(F)=\bigl\{\alpha\in\widehat{A}\;:\;
  h^i(A,F\otimes P_\alpha)>0\bigr\}.
\]
It is a closed subset of $\widehat{A}$ by semicontinuity.
\begin{enumerate}
\item $F$ satisfies the \emph{index theorem with index $0$}
  (\emph{IT$_0$} for short) if $V^i(F)=\emptyset$ for all $i>0$; that is,
  $H^i(A,F\otimes P_\alpha)=0$ for all $\alpha\in\widehat{A}$ and all $i>0$.
\item $F$ is \emph{M-regular} if
  $\operatorname{codim}_{\widehat{A}}V^i(F)>i$ for all $i>0$.
\item $F$ is a \emph{GV-sheaf} (generic vanishing sheaf) if
  $\operatorname{codim}_{\widehat{A}}V^i(F)\ge i$ for all $i>0$.
\end{enumerate}
Here the empty set has codimension $+\infty$. Thus
\[
  \text{IT}_0\;\Longrightarrow\;\text{M-regular}\;\Longrightarrow\;\text{GV}.
\]
\end{definition}

\begin{remark}\label{rem:Mreg-original}
In \cite{PP1}, M-regularity is defined by the condition
\[
  \operatorname{codim}_{\widehat{A}}\,
  \operatorname{Supp} R^i\Phi_{\mathcal{P}}(F)>i
  \qquad\text{for all } i>0,
\]
where $R^i\Phi_{\mathcal{P}}(F)$ is the $i$-th cohomology sheaf of the
Fourier--Mukai transform $\Phi_{\mathcal{P}}(F)$. This is equivalent to
condition~(2) of Definition~\ref{def:GV}; see \cite[Lemma~3.6]{PP-GV}. We use
the formulation in terms of the loci $V^i(F)$ because it is the one that
enters our arguments.
\end{remark}

We shall use the following result of Pareschi and Popa on tensor products.

\begin{proposition}[Pareschi--Popa \cite{PP1}] \label{prop:pareschi_popa}
If $\mathcal{F}$ is an $M$-regular coherent sheaf and $\mathcal{H}$ is a locally free sheaf satisfying I.T. with index 0, then their tensor product $\mathcal{F} \otimes \mathcal{H}$ also satisfies I.T. with index 0.
\end{proposition}

\subsection*{Connecting \texorpdfstring{$M$}{M}-Regularity to ACM Bundles}
Using the Pareschi-Popa framework, we can build a direct bridge to the ACM property.

\begin{proposition}[An ACM criterion from M-regularity]
\label{prop:Mregular-ACM}
Let $(A,L)$ be a polarized abelian variety of dimension
$g\geq 2$, and let $F$ be a vector bundle on $A$.
Assume that both $F$ and $F^\vee\otimes L$ are
M-regular. Then $F$ is ACM with respect to $L$
if and only if
\[
H^i(A,F)=H^i(A,F^\vee\otimes L)=0
\qquad\text{for every }0<i<g.
\]
\end{proposition}

\begin{proof}
Suppose first that $F$ is ACM. The vanishing of
$H^i(A,F)$ follows by taking the twist $0$.
Moreover, Serre duality gives
\[
H^i(A,F^\vee\otimes L)
\simeq
H^{g-i}(A,F\otimes L^{-1})^\vee=0
\qquad(0<i<g).
\]

Conversely, assume the stated vanishings.
By the Pareschi--Popa tensor-product theorem,
an M-regular sheaf tensored with an ample line
bundle satisfies $\mathrm{IT}_0$.
Therefore
\[
H^i(A,F\otimes L^{\otimes m})=0
\qquad(m\geq1,\ i>0).
\]

For $m\leq-2$, Serre duality yields
\[
H^i(A,F\otimes L^{\otimes m})
\simeq
H^{g-i}\!\left(
A,(F^\vee\otimes L)\otimes L^{\otimes(-m-1)}
\right)^\vee.
\]
Since $-m-1\geq1$, the sheaf on the right satisfies
$\mathrm{IT}_0$. Hence this group vanishes for
$0<i<g$.

The twist $m=0$ is covered by hypothesis.
For $m=-1$, Serre duality gives
\[
H^i(A,F\otimes L^{-1})
\simeq
H^{g-i}(A,F^\vee\otimes L)^\vee=0
\qquad(0<i<g).
\]
Thus $F$ is ACM.
\end{proof}

\begin{remark}[M-regularity alone does not imply ACM]
Let $(A,L)$ be a polarized abelian variety of dimension
$g\geq2$. The ample line bundle $L$ satisfies
$\mathrm{IT}_0$ and is therefore M-regular. Nevertheless,
it is not ACM with respect to $L$, since
\[
H^1(A,L\otimes L^{-1})
=
H^1(A,\mathcal O_A)
\neq0.
\]
Thus even $\mathrm{IT}_0$ does not imply the ACM
property. In Proposition~\ref{prop:Mregular-ACM},
the M-regularity assumptions control the twists
$m\geq1$ and $m\leq-2$, while the additional
vanishings control the boundary twists $m=0,-1$.
\end{remark}

\begin{theorem}\label{thm:GV-ACM}
Let $E$ be a locally free sheaf on an abelian variety $A$ of dimension $g\ge 2$
such that $E$ and $E^\vee$ are GV. Set
\[
W \;=\; \widehat{A}\smallsetminus \bigcup_{i=1}^{g-1} V^i(E),
\]
a nonempty open subset of $\widehat{A}$. Then for every $\alpha\in W$ and every
ample line bundle $L$ on $A$, the bundle $E\otimes P_\alpha$ is ACM with respect
to $L$. In particular, if $H^i(A,E)=0$ for $0<i<g$, then $E$ itself is ACM with
respect to every ample $L$.
\end{theorem}

\begin{proof}
For \(1\leq i\leq g-1\), the GV condition makes \(V^i(E)\) a
proper closed subset of \(\widehat A\). Hence \(W\) is a nonempty
open subset. Fix \(\alpha\in W\), an ample line bundle \(L\), and
an integer \(i\) with \(0<i<g\).

If \(m\geq 1\), then \(P_\alpha\otimes L^{\otimes m}\) is an ample
line bundle and satisfies \(\mathrm{IT}_0\). Since \(E\) is GV,
Pareschi and Popa's tensor-product theorem
\cite[Proposition~3.1]{PP3} implies that
\(E\otimes P_\alpha\otimes L^{\otimes m}\) satisfies
\(\mathrm{IT}_0\). In particular, its \(i\)-th cohomology vanishes.
For \(m=0\), the vanishing follows from \(\alpha\in W\).

If \(m=-t\) with \(t\geq 1\), Serre duality and
\(\omega_A\simeq\mathcal O_A\) give
\[
    H^i(A,E\otimes P_\alpha\otimes L^{\otimes(-t)})^\vee
      \simeq
    H^{g-i}(A,E^\vee\otimes P_{-\alpha}\otimes L^{\otimes t}).
\]
The same tensor-product theorem applies to the GV sheaf \(E^\vee\)
and the ample line bundle \(P_{-\alpha}\otimes L^{\otimes t}\).
The right-hand side vanishes since \(g-i>0\).
This proves the ACM property for every ample \(L\). If the
intermediate cohomology of \(E\) vanishes, then \(0_{\widehat A}\)
belongs to \(W\), which gives the last assertion.
\end{proof}

\subsection*{Does the ACM Property Imply \texorpdfstring{$M$}{M}-Regularity?}

A natural question arises: does the Arithmetically Cohen-Macaulay (ACM) condition automatically guarantee $M$-regularity? The answer is no. 

\begin{proposition}[ACM does not imply M-regularity]\label{prop:acm-not-mreg}
Let $(A,L)$ be a polarized abelian variety of dimension
$g\geq1$. Every nontrivial line bundle
$P\in\operatorname{Pic}^0(A)$ is ACM with respect to $L$,
but is not M-regular.
\end{proposition}

\begin{proof}
The ACM property of $P$ follows from the vanishing
results established above. On the other hand,
\[
V^g(P)
=
\left\{
Q\in\operatorname{Pic}^0(A)
\mid H^g(A,P\otimes Q)\neq0
\right\}
=
\{P^{-1}\}.
\]
Indeed, $P\otimes Q$ is acyclic unless
$Q\simeq P^{-1}$, whereas
\[
H^g(A,\mathcal O_A)\simeq k.
\]
Consequently,
\[
\operatorname{codim}_{\operatorname{Pic}^0(A)}V^g(P)
=g,
\]
which violates the strict inequality required for
M-regularity.
\end{proof}

\subsection{Support loci of twisted unipotent bundles and slope-stable semihomogeneous ACM bundles}
We conclude by placing the two constructions of the paper in the Pareschi--Popa hierarchy.
Proposition~\ref{prop:unip-gv} treats the bundles \(P\otimes U\).
The semihomogeneous bundles of Theorem~\ref{thm:stable-acm-semihomogeneous} give a complementary stable picture: a positive rational slope yields an IT\(_0\) bundle, while a negative rational slope yields a slope-stable ACM bundle which is not a GV-sheaf.

\begin{proposition}[Position of Twisted Unipotent Bundles in the GV/$M$-Regularity Hierarchy]\label{prop:unip-gv}
Let $(A, L)$ be a polarized abelian variety of dimension $g \ge 2$, let $P = P_\alpha \in \operatorname{Pic}^0(A) \setminus \{\mathcal{O}_A\}$, and let $U \in \operatorname{Unip}(A)$ be a nonzero unipotent vector bundle. Set $E := P_\alpha \otimes U$. Then:
\begin{enumerate}
    \item[(1)] For every $0 \le i \le g$, the cohomological support locus of $E$ is a single point:
    \[
        V^i(E) = \{-\alpha\} \subset \hat{A}.
    \]
    \item[(2)] Both $E$ and $E^\vee$ are generic vanishing ($GV$) sheaves on $A$, so $E$ satisfies the hypotheses of Theorem~\ref{thm:GV-ACM} with $0_{\hat{A}} \notin \bigcup_{i=1}^{g-1} V^i(E) = \{-\alpha\}$.
    \item[(3)] The bundle $E$ is ACM with respect to $L$ and satisfies $\operatorname{codim}_{\hat{A}} V^i(E) = g > i$ for all $1 \le i \le g-1$, but fails to be $M$-regular at top degree $i = g$ since $\operatorname{codim}_{\hat{A}} V^g(E) = g \not> g$.
\end{enumerate}
\end{proposition}
\begin{proof}
For every $\beta\in\widehat A$, we have
\[
E\otimes P_\beta \simeq P_{\alpha+\beta}\otimes U.
\]
If $\beta\neq-\alpha$, then $P_{\alpha+\beta}$ is nontrivial,
so Proposition~\ref{prop:unip-cohom-reg}(1), applied at twist zero, gives
$H^i(A,E\otimes P_\beta)=0$ for every $i$.
If $\beta=-\alpha$, then $E\otimes P_\beta\simeq U$, and
Theorem~\ref{thm:torsion-dichotomy}(1) gives $H^i(A,U)\neq0$ for every $0\leq i\leq g$.
Hence
\[
V^i(E)=\{-\alpha\}
\qquad (0\leq i\leq g).
\]

Since $U^\vee$ is unipotent and
$E^\vee\simeq P_{-\alpha}\otimes U^\vee$, the same argument gives
$V^i(E^\vee)=\{\alpha\}$ for every $0\leq i\leq g$.
Both singletons have codimension $g$ in $\widehat A$.
Thus $E$ and $E^\vee$ are GV. Moreover, $\alpha\neq0$, so
\[
0_{\widehat A}\notin\bigcup_{i=1}^{g-1}V^i(E).
\]
Applying Theorem~\ref{thm:GV-ACM} at $0_{\widehat A}$, we conclude that
$E$ is ACM with respect to $L$.

Finally, $\operatorname{codim}_{\widehat A}V^i(E)=g>i$
for $1\leq i<g$, while
$\operatorname{codim}_{\widehat A}V^g(E)=g$.
Thus $E$ fails the M-regularity condition precisely in degree $g$.
\end{proof}

\begin{proposition}
\label{prop:semihom-gv}
Let \((A,L)\) be a polarized abelian variety of dimension \(g\ge 2\) over an algebraically closed field of characteristic zero.
Let
\[
q=\frac{a}{b}\in\mathbb{Q}\setminus\mathbb{Z},
\qquad b\ge 2,\qquad \gcd(a,b)=1,
\]
and let \(E\) be a simple semihomogeneous vector bundle on \(A\) with
\[
\delta(E)=\frac{c_1(E)}{\operatorname{rk}(E)}=q\,c_1(L)
\quad\text{in }\operatorname{NS}(A)_{\mathbb{Q}}.
\]
Then \(E\) is \(\mu_L\)-stable and arithmetically Cohen--Macaulay with respect to \(L\).
Moreover:
\begin{enumerate}
\item If \(q>0\), then \(E\) satisfies IT\(_0\). In particular \(E\) is \(M\)-regular.
\item If \(q<0\), then
\[
V^i(E)=\emptyset
\quad\text{for }1\le i\le g-1,
\qquad
V^g(E)=\widehat{A}.
\]
Thus \(\operatorname{codim}_{\widehat{A}} V^g(E)=0\), so \(E\) is not a GV-sheaf and, in particular, not \(M\)-regular.
\end{enumerate}
\end{proposition}

\begin{proof}
Stability and the ACM property are Theorem~\ref{thm:stable-acm-semihomogeneous}.
The same proof shows that a simple semihomogeneous bundle whose rational slope class is ample satisfies IT\(_0\): Mukai's isogeny description realizes it as \(\pi_*M\) with \(M\) ample, and Kodaira vanishing on the source, together with \(\omega\simeq\mathcal{O}\), gives the required vanishing after every \(\operatorname{Pic}^0\)-twist.

Suppose first that \(q>0\). Then \(\delta(E)=q\,c_1(L)\) is an ample class in \(\operatorname{NS}(A)_{\mathbb{Q}}\). The observation above yields IT\(_0\), and IT\(_0\) implies \(M\)-regularity.

Suppose now that \(q<0\). Then
\[
\delta(E^\vee)=-q\,c_1(L)
\]
is ample, so \(E^\vee\) satisfies IT\(_0\). In the isogeny description, \(E^\vee\cong\pi_*M\) with \(M\) ample, hence
\[
\chi(A,E^\vee)=\chi(B,M)=\frac{c_1(M)^g}{g!}>0.
\]
For every \(\alpha\in\widehat{A}\), the bundle \(E^\vee\otimes P_\alpha\) is a \(\operatorname{Pic}^0\)-twist of an IT\(_0\) bundle, so
\[
h^0(A,E^\vee\otimes P_\alpha)=\chi(A,E^\vee)>0.
\]
Serre duality and \(\omega_A\simeq\mathcal{O}_A\) give
\[
H^i(A,E\otimes P_\beta)^\vee
\simeq
H^{g-i}(A,E^\vee\otimes P_{-\beta}).
\]
For \(1\le i\le g-1\) one has \(g-i>0\), so the right-hand side vanishes and \(V^i(E)=\emptyset\).
For \(i=g\) the same identity gives
\[
h^g(A,E\otimes P_\beta)=h^0(A,E^\vee\otimes P_{-\beta})>0
\]
for every \(\beta\). Therefore \(V^g(E)=\widehat{A}\).
\end{proof}
\begin{remark}
Let \(E\) be as in Proposition~\ref{prop:semihom-gv}. For \(q>1\), the bundle \(E\)
is ACM by Theorem~\ref{thm:stable-acm-semihomogeneous} and satisfies
\(\mathrm{IT}_0\) by Proposition~\ref{prop:semihom-gv}(1). The simple semihomogeneous
bundle \(E^\vee\otimes L\) has rational slope class
\[
    \delta(E^\vee\otimes L)=(1-q)c_1(L).
\]
Since \(1-q<0\), Proposition~\ref{prop:semihom-gv}(2) shows that \(E^\vee\otimes L\)
fails the GV condition and consequently fails \(M\)-regularity.
Thus an ACM bundle can fail the \(M\)-regularity hypotheses of
Proposition~\ref{prop:Mregular-ACM}.

If \(q<0\), Proposition~\ref{prop:semihom-gv}(2) gives \(V^g(E)=\widehat A\).
This gives slope-stable ACM bundles that fail the GV condition.
\end{remark}
\appendix

\section{Fourier-Mukai Equivalence for Unipotent Bundles}

Let $A$ be an abelian variety of dimension $g$ over an algebraically closed field $k$, and let $\widehat{A}$ be its dual abelian variety. Let $\mathcal{P}$ denote the normalized Poincaré line bundle on $A \times \widehat{A}$, satisfying $\mathcal{P}|_{A \times \{\hat{0}\}} \cong \mathcal{O}_A$ and $\mathcal{P}|_{\{0\} \times \widehat{A}} \cong \mathcal{O}_{\widehat{A}}$. 

Let $p_A : A \times \widehat{A} \rightarrow A$ and $p_{\widehat{A}} : A \times \widehat{A} \rightarrow \widehat{A}$ be the canonical projections. The standard Fourier-Mukai transforms between the bounded derived categories are defined as:
\begin{align*}
    \Phi(E) &= R p_{\widehat{A},*} (p_A^* E \otimes \mathcal{P}) : D^b(A) \rightarrow D^b(\widehat{A}) \\
    \Psi(F) &= R p_{A,*} (p_{\widehat{A}}^* F \otimes \mathcal{P}) : D^b(\widehat{A}) \rightarrow D^b(A)
\end{align*}

Let $\Unip(A)$ denote the abelian category of unipotent vector bundles on $A$ (vector bundles admitting a finite filtration whose successive quotients are isomorphic to $\mathcal{O}_A$). Let \(\operatorname{Coh}_{\{0\}}(\widehat A)\) denote the category
of coherent sheaves on \(\widehat A\) whose set-theoretic support
is contained in \(\{0_{\widehat A}\}\).

\begin{theorem}[Fourier--Mukai Equivalence for Unipotent Bundles]
\label{thm:appendix-equivalence}
Let $A$ be an abelian variety of dimension $g$ over an algebraically closed 
field $k$, and let $\widehat{A} = \mathrm{Pic}^0(A)$ be its dual abelian variety. 
There is a chain of equivalences of abelian categories
\[
\mathrm{Unip}(A) 
\;\simeq\; 
\mathrm{Coh}_{\{0\}}(\widehat{A}) 
\;\simeq\; 
\mathrm{flmod}\!\left(\mathcal{O}_{\widehat{A},\,0}\right) 
\;\simeq\; 
\mathrm{flmod}\!\left(k[[x_1,\ldots,x_g]]\right).
\]
More precisely:
\begin{enumerate}
    \item[\rm(i)] The shifted Fourier-Mukai functors $\Phi^g := \mathcal{H}^g \circ \Phi$ and $\Psi^0 := (-1_A)^* \circ \mathcal{H}^0 \circ \Psi$ are mutually inverse exact equivalences
$$
\Phi^g \colon \operatorname{Unip}(A) \longrightarrow \operatorname{Coh}_{\{0\}}(\widehat{A}), \quad \Psi^0 \colon \operatorname{Coh}_{\{0\}}(\widehat{A}) \longrightarrow \operatorname{Unip}(A).
$$
  \item[\rm(ii)] The stalk functor $F \mapsto F_0$ is an 
    equivalence
    \[
    \mathrm{Coh}_{\{0\}}(\widehat{A}) \;\xrightarrow{\;\sim\;}\; 
    \mathrm{flmod}\!\left(\mathcal{O}_{\widehat{A},\,0}\right).
    \]
    \item[\rm(iii)] Since $\widehat{A}$ is smooth of dimension $g$, the local ring 
    $\mathcal{O}_{\widehat{A},0}$ is a regular local ring of dimension $g$, and 
    the Cohen Structure Theorem gives $\widehat{\mathcal{O}}_{\widehat{A},0} \cong 
    k[[x_1,\ldots,x_g]]$. As every finite-length module is complete, this yields
    \[
    \mathrm{flmod}\!\left(\mathcal{O}_{\widehat{A},\,0}\right) 
    \;\simeq\; 
    \mathrm{flmod}\!\left(k[[x_1,\ldots,x_g]]\right).
    \]
\end{enumerate}
\end{theorem}

\begin{proof}
The equivalence in (i) is Mukai's equivalence for unipotent
bundles \cite[Theorem~4.12]{MukaiSemiHomogeneous}. To see the
restriction of the transforms, start from
\(\Phi(\mathcal O_A)\simeq k(0_{\widehat A})[-g]\).
Applying \(\Phi\) to the filtration of a unipotent bundle \(U\)
shows by induction that \(\Phi(U)\) is concentrated in degree \(g\).
Its degree-\(g\) cohomology sheaf is an iterated extension of
\(k(0_{\widehat A})\), so it has finite length and is supported
set-theoretically at the origin.

Conversely, a sheaf \(\mathcal F\) of finite length supported at
\(0_{\widehat A}\) has a composition series with factors
\(k(0_{\widehat A})\). The normalization of the Poincar\'e bundle
gives \(\Psi(k(0_{\widehat A}))\simeq\mathcal O_A\).
Applying \(\Psi\) to the composition series shows by induction
that \(\Psi(\mathcal F)\) is concentrated in degree zero and is a
successive extension of copies of \(\mathcal O_A\). These
extensions are locally free, so the resulting bundle is unipotent.
The functors \(\mathcal H^g\circ\Phi\) and \(\mathcal H^0\circ\Psi\)
are exact on the respective subcategories by the long exact
sequences of cohomology sheaves. Mukai's inversion formula
\cite{MukaiDuality} identifies \((-1_A)^*\circ\Psi[g]\) with the
derived inverse of \(\Phi\). Consequently
\(\Psi_0=(-1_A)^*\circ\mathcal H^0\circ\Psi\) is inverse to
\(\Phi^g=\mathcal H^g\circ\Phi\) on the categories in (i).

To establish the remaining equivalences, we pass to local algebra. For any coherent sheaf supported entirely at the closed point $0_{\widehat{A}}$, its global module structure is entirely determined by its stalk. The stalk functor $F \mapsto F_0$ therefore yields an exact equivalence:
\[
\mathrm{Coh}_{\{0\}}(\widehat{A}) \simeq \mathrm{flmod}(\mathcal{O}_{\widehat{A},0}).
\]

Finally, since the dual abelian variety $\widehat{A}$ is smooth of dimension $g$ over an algebraically closed field, the local ring $R := \mathcal{O}_{\widehat{A},0}$ is a regular local ring of dimension $g$. By the Cohen Structure Theorem, its completion at the maximal ideal $\mathfrak{m}_0$ is isomorphic to the formal power series ring:
\[
\hat{R} \cong k[[x_1, \dots, x_g]].
\]
Because every finite-length module $M$ over a Noetherian local ring is annihilated by some power of the maximal ideal, it is naturally complete, meaning $M \otimes_R \hat{R} \cong M$. Consequently, the base change functor induces a canonical exact equivalence:
\[
\mathrm{flmod}(\mathcal{O}_{\widehat{A},0}) \simeq \mathrm{flmod}(k[[x_1, \dots, x_g]]).
\]
Composing these functors yields the stated chain of equivalences.
\end{proof}

\begin{theorem}
\label{thm:flmod_equivalence}
Let $k$ be a field. The category $\operatorname{flmod}(k[[x_1, x_2]])$ of finite-length modules over the formal power series ring $k[[x_1, x_2]]$ is equivalent to the category of finite-dimensional $k$-vector spaces equipped with a pair of commuting nilpotent endomorphisms.
\end{theorem}

\begin{proof}
Let $\mathcal{C}$ denote the category whose objects are tuples $(V, X, Y)$, where $V$ is a finite-dimensional $k$-vector space and $X, Y \in \operatorname{End}_k(V)$ satisfy $XY = YX$ and $X^N = Y^N = 0$ for some integer $N \ge 1$. A morphism $f \colon (V, X, Y) \to (V', X', Y')$ in $\mathcal{C}$ is a $k$-linear map $f \colon V \to V'$ such that $f \circ X = X' \circ f$ and $f \circ Y = Y' \circ f$. We construct mutually inverse exact functors between $\operatorname{flmod}(k[[x_1, x_2]])$ and $\mathcal{C}$.

\textbf{Forward Direction:} Let $M \in \operatorname{flmod}(k[[x_1, x_2]])$. Because $M$ has finite length and $k[[x_1, x_2]]$ is a $k$-algebra, $M$ is naturally a finite-dimensional $k$-vector space. Let $X, Y \in \operatorname{End}_k(M)$ be the linear transformations defined by scalar multiplication by $x_1$ and $x_2$, respectively. Since the base ring is commutative, $XY = YX$. Furthermore, $k[[x_1, x_2]]$ is a Noetherian local ring with a unique maximal ideal $\mathfrak{m} = \langle x_1, x_2 \rangle$. Because $M$ has finite length, it is annihilated by some power of the maximal ideal; that is, there exists an integer $N \ge 1$ such that $\mathfrak{m}^N M = 0$. Since $x_1, x_2 \in \mathfrak{m}$, this forces $X^N = 0$ and $Y^N = 0$. Any module homomorphism over $k[[x_1, x_2]]$ naturally commutes with multiplication by $x_1$ and $x_2$, establishing a well-defined functor into $\mathcal{C}$.

\textbf{Reverse Direction:} Conversely, let $(V, X, Y) \in \mathcal{C}$. The commuting operators $X$ and $Y$ endow $V$ with the structure of a $k[x_1, x_2]$-module via polynomial evaluation: $p(x_1, x_2) \cdot v = p(X, Y)(v)$. To extend this action to the completion $k[[x_1, x_2]]$, consider a formal power series $s(x_1, x_2) = \sum_{i,j=0}^{\infty} c_{i,j} x_1^i x_2^j$. Because $X$ and $Y$ are nilpotent and commute, any term with $i \ge N$ or $j \ge N$ acts as the zero operator on $V$. Consequently, the infinite series evaluates to a finite sum in $\operatorname{End}_k(V)$, defining a unique and well-defined $k[[x_1, x_2]]$-module structure on $V$. Since $V$ is finite-dimensional over $k$, it inherently possesses finite length as a module. Any morphism in $\mathcal{C}$ preserves this action by definition, yielding a functor into $\operatorname{flmod}(k[[x_1, x_2]])$.

These two constructions preserve exactness and are mutually inverse by definition, establishing the equivalence of the categories.
\end{proof}

\section*{Acknowledgements}
The first author was supported by the Indian Institute of Science Education and Research Berhampur (IISER Berhampur) through his appointment as a Postdoctoral Research Fellow (PDRF).

\section*{Declarations}

\textbf{Use of AI tools.} During the preparation of this work, the authors used ChatGPT solely for language refinement and to assist in checking routine results. After using this tool, the authors reviewed and edited the manuscript and take full responsibility for the final content.

\end{document}